\documentclass{amsart}
\usepackage{geometry,graphicx,amssymb,amsmath,amsbsy,eucal,amsfonts,mathrsfs,amscd,bm,tcolorbox, subcaption, enumitem}
\usepackage{tikz}
\usetikzlibrary{patterns}
\usetikzlibrary{cd}
\usepackage{cleveref}
\usepackage{caption}
\usepackage{subcaption}

\usepackage[normalem]{ulem}

\numberwithin{equation}{section}

\allowdisplaybreaks[3]

\newtheorem{theorem}{Theorem}[section]
\newtheorem{lemma}[theorem]{Lemma}
\newtheorem{assumption}[theorem]{Assumption}

\newtheorem{corollary}[theorem]{Corollary}
\newtheorem{proposition}[theorem]{Proposition}

\theoremstyle{definition}
\newtheorem{definition}[theorem]{Definition}
\newtheorem{example}{Example}

\theoremstyle{remark}
\newtheorem{remark}[theorem]{Remark}

\theoremstyle{plain}
\newtheorem*{proposition*}{Proposition}

\DeclareMathOperator{\sgn}{sgn}

\def\id{\mathrm{id}}
\def\hook{\;\lrcorner\;}

\usepackage{mydef}

\begin{document}

\title[Discrete Poincar\'e and Bogovski\u{\i} operators]{Discrete Poincar\'e and Bogovski\u{\i} operators on\\ cochains and Whitney forms}

\author{Johnny Guzm\'an}
\address{Division of Applied Mathematics
Brown University
Box F
182 George Street
Providence, RI 02912, USA}
\email{johnny\_guzman@brown.edu}

\author{Anil N. Hirani}
\address{Department of Mathematics,
University of Illinois at Urbana-Champaign,
1305 West Green Street,
Urbana, IL 61801, USA}
\email{hirani@illinois.edu}

\author{Bingyan Liu}
\address{Department of Mathematics, 
University of Illinois at Urbana-Champaign,
1305 West Green Street,
Urbana, IL 61801, USA}
\curraddr{165 W Superior St, Chicago, IL 60654, IL}
\email{bingyanliu21@gmail.com}

\author{Pratyush Potu}
\address{Division of Applied Mathematics
Brown University
Box F
182 George Street
Providence, RI 02912, USA}
\email{pratyush\_potu@brown.edu}


\begin{abstract}
Smooth Poincar\'e operators are a tool used to show the vanishing of smooth de Rham cohomology on contractible manifolds and have found use in the analysis of finite element methods based on the Finite Element Exterior Calculus (FEEC). We construct analogous discrete Poincar\'e operators acting on cochains and Whitney forms. We provide explicit, constructive realizations of these operators under various assumptions on the underlying domain or simplicial complex. In particular, we provide simple constructions for the discrete Poincar\'e operators on simplicial complexes which are collapsible and those with underlying domain being star-shaped with respect to a point. We then provide more abstract constructions on simplicial complexes which are discrete contractible and domains which are Lipschitz contractible. We also modify the discrete Poincar\'e operator on star-shaped domains to construct a discrete Bogovski\u{\i} operator which satisfies the requisite homotopy identity while preserving homogeneous boundary conditions. Applications arise in the construction of discrete scalar and vector potentials and in the discrete wedge product of Discrete Exterior Calculus (DEC).
\end{abstract}

\maketitle


\section{Introduction}

A key construction in the proof of the smooth homotopy invariance of de Rham cohomology is the homotopy operator. In the numerical analysis literature, a special case of this operator is called the Poincar\'e map or operator~\cite{Demkowicz2003, Demkowicz2005, Costabel2010, CapHu2023}. See Section~\ref{subsec:smooth-poincare} for the smooth construction and definitions.

The utility of Poincar\'e operators in the analysis of finite elements was first realized in \cite{Hiptmair1999}. Thereafter, Poincar\'e operators which preserve Whitney forms \cite{ArnoldFalkWinther2006, Whitney1957} have been used in the construction of finite element spaces with structure-preserving properties \cite{ArnoldFalkWinther2006, ChristiansenHu2018, Hiptmair1999}. Also valuable are Poincar\'e operators which preserve homogenous boundary conditions. These are referred to as Bogovski\u{\i} operators in reference to the right inverse of the divergence constructed in \cite{Bogovskii1979}. Of particular note are the regularized Poincar\'e and Bogovski\u{\i} operators of \cite{Costabel2010}, which are bounded between a variety of function spaces, allowing for the establishment of Poincar\'e inequalities.

A natural question is whether one can define an analogous \textit{discrete} operator directly defined over cochains (as they are isomorphic to the Whitney forms). Explicitly, given a cochain complex $(\C^\bullet, \dd^\bullet)$, a \textit{discrete} Poincar\'e operator is an operator, $P:\C^k \to \C^{k-1}$, satisfying
\begin{subequations}\label{eq:homotopy_formula}
    \begin{alignat}{2}
    (\dd P + P \dd) \alpha &= \alpha, \qquad && \alpha\in \C^k, k \geq 1,\\
    P\dd \alpha &= \alpha - \pi \alpha, \qquad && \alpha \in \C^0,
\end{alignat}
\end{subequations}
where $\pi$ is an appropriate constant map. The existence of such an operator allows one to establish the vanishing of cohomology spaces $H^k(\C^\bullet)$ for $k>0$. Moreover, such an operator provides a numerical method for finding discrete scalar and vector potentials (i.e. solutions of the differential equation $du = f$ for appropriate $f$). Having such a numerical method is useful in simulation of magnetostatic \cite{LeMenachClenetPiriou1998, Rodriguez_magnetostatics_2013} and eddy-current \cite{Rodriguez_eddycurrent_2015} problems. Another more recent application is in discrete de Rham methods where discrete potentials enable the construction of conforming liftings which in turn allow for proofs of consistency \cite{dipietro2025}.

Another application is in discrete exterior calculus (DEC) where a discrete wedge product on cochains is defined as an anti-symmetrized cup product~\cite[Chapter 7]{Hirani2003}. This product is not associative at the level of cochains but only at the level of cohomology. Thus the cochain complex with this wedge product does not form a differential graded algebra (DGA) even though it is commutative and satisfies a graded Leibniz property. In \cite{Stasheff1963}, the defect in associativity was resolved by introduction of an infinite sequence of operators in the context of loop spaces. Such structures are now called $A_\infty$-algebras. It is shown in~\cite{Liu2026} that a discrete Poincar\'e operator is a key step in a recursive construction of a $A_\infty$ structure on cochains for contractible domains following a technique in~\cite{DoMoSh2008}. See~\cite{Liu2026} for details.

Prior work on discrete Poincar\'e operators on triangulations of contractible domains is in \cite{Desbrun2005} (see also \cite{Leok2004}) where such operators are constructed for simplicial complexes that can be grown by a process of local augmentation by trivially star-shaped complexes. This condition is referred to as being a generalized star-shaped complex in \cite{Desbrun2005}, and is in fact equivalent to being what is referred to as strongly collapsible (or LC-reducible) in~\cite{Matousek08, Barmak2012}. Other work relevant to discrete Poincar\'e operators is in the literature on numerical methods for the computation of discrete scalar and vector potentials. See~\cite{Pitassi2022} and \cite{RodriguezValli2015} and references therein which focus on the case of simplicial complexes embedded in $\R^3$, as well as~\cite{dipietro2025} which builds on~\cite{Pitassi2022} to construct discrete potentials in arbitrary dimension. We also note the work of \cite{ErnVohralik20} and references therein which construct potentials by solving local problems on ``shellable'' patches.

In this work, we introduce discrete Poincar\'e operators for complexes of cochains and Whitney forms under different sets of assumptions. Unlike some of the previous work noted above, we do not investigate the boundedness of our discrete Poincar\'e operators. We generally consider two points of view in the definition of these operators: one which is based on singular simplices and made discrete through the use of Whitney forms, and one which is purely simplicial working directly with cochains. The operators constructed from both perspectives are analogous in their constructions, but rely on different assumptions.  

First, we provide a discrete Poincar\'e operator on polytopal subsets of $\R^n$ which are star-shaped with respect to a point (see Section \ref{sec:discPoincWhitsimple}). The operator is based on integration of Whitney forms over simplicial regions (cones). These cones generally will not coincide with facets in the simplicial complex and so implementation of the operator requires numerical integration over cut mesh elements. 
We then modify this approach to obtain a discrete Bogovski\u{\i} operator which preserves homogeneous boundary conditions (see Section \ref{sec:Bog}). 
We also generalize the previous discrete Poincar\'e operator to domains which are contractible: there exists a homotopy contracting the domain to a point (see Section \ref{sec:Whitney_based_discrete_op}). The resulting discrete Poincar\'e operator now requires integrating Whitney forms over curved regions. We remark that while these operators are defined over Whitney forms, we naturally obtain a discrete Poincar\'e operator over cochains using the Whitney and de Rham maps. We note that the numerical integration required for these operators may be challenging, but has been investigated in \cite{Farrell2011} for example. We also provide some details for the numerical implementation in 2d in Section \ref{sec:numerical_experiments}.

To avoid complicated integration, it is desirable to have a ``combinatorial'' discrete Poincar\'e operator defined directly over cochains. First, we show that this is possible when the simplicial complex of the domain admits a ``discrete contraction:'' there exists a homotopy contracting the domain to a point which is a simplicial map (see Section \ref{sec:simp_cones}).  Finding such a discrete contraction may be difficult, but we provide an explicit construction when the simplicial complex is strongly collapsible (see Section \ref{sec:strong_collapse_disc_contr}). Finally, we describe a remarkable discrete Poincar\'e operator when the simplicial complex is collapsible (see Section \ref{sec:collapse_cone}) which was independently discovered in the works of \cite{Desbrun2005} and \cite{Pitassi2022}. We show how this operator arises naturally from the foundational work of Whitehead \cite{Whitehead1950} in Appendix \ref{app:Whitehead}.

The paper is organized as follows. In Section~\ref{sec:prelim}, we recall some basic notions of simplicial complexes and smooth Poincar\'e operators and set the notation for future sections. In Section~\ref{sec:discPoincidea} we show how the discrete Poincar\'e operators arise from generalized cone operators and hint at an alternate point of view. Then, in Section \ref{sec:cones_special_cases}, we show some simple examples of generalized cone operators under certain assumptions on the simplicial complex or underlying domain. Following this, in Section \ref{sec:cones_general}, we show a general framework for constructing these generalized cone operators and provide some concrete realizations of this framework in Section \ref{sec:discPoincHom}.  Then, in Section \ref{sec:Bog}, we show to construct the discrete Bogovski\u{\i} operators on domains which are star-shaped with respect to a point by modifying a previously defined discrete Poincar\'e operator. Finally, we illustrate our constructed operators with numerical experiments in Section \ref{sec:numerical_experiments}.

\section{Preliminaries}\label{sec:prelim}

We assume familiarity with basic concepts from simplicial complexes, exterior calculus and Whitney forms. Some standard references are \cite{Munkres1984} for simplicial theory,~\cite{Lee2013,AbMaRa1988} for exterior calculus, and \cite{Whitney1957,Dodziuk1976,ArnoldFalkWinther2006} for Whitney forms.

To fix notation, let $X$ denote a geometric simplicial complex with underlying domain $\Omega = |X|\subset\R^n$. We denote by $X^{k}$ the subset of $X$ consisting of all the $k$-simplices in $X$. The vector spaces of $k$-chains and $k$-cochains with coefficients in $\R$ are denoted $\C_k(X)$ and $\C^k(X)$ respectively. We may simply write $\C_k$ or $\C^k$ when the simplicial complex is clear. We let $\partial_k:\C_k\to \C_{k-1}$ be the boundary map and $\dd^k:\C^k\to \C^{k+1}$ be the coboundary map. The duality pairing between chains and cochains is denoted $\bl \cdot, \cdot \br$.

We let $\Lambda^k(\Omega)$ denote the space of smooth differential $k$-forms on $\Omega$ where $k\in\{0, \dots, n\}$. We recall the Stokes' theorem for a simplicial $k$-chain $c\in \C_k$ and smooth enough differential $k$-form $\omega$:
\begin{equation}\label{eq:StokesSimp}
    \int_c d^{k-1} \omega = \int_{\partial_k c} \omega.
\end{equation}

The space of Whitney $k$-forms is denoted $V_h^k$. The de Rham map is denoted $\dr^k$ and maps a sufficiently smooth $k$-form (which includes $V_h^k$) to a $k$-cochain. The Whitney map is denoted $W^k: \C^k \rightarrow V_h^k$ and we note the following properties of these two maps.
\begin{subequations}
\begin{alignat}{1}
d^k W^k=&  W^{k+1} \dd^{k}, \label{aux1}\\
\dd^k \dr^k=& \dr^{k+1} d^k,  \label{aux2}\\
\dr^k W^k=&\text{id}. \label{aux3}
\end{alignat}
\end{subequations}
One can see \cite[Chapter IV, Section 27]{Whitney1957} for the first and third identity and the second identity is Stokes' theorem on simplices (i.e. \eqref{eq:StokesSimp}). 

\subsection{Smooth Poincar\'e operators}
\label{subsec:smooth-poincare}
Smoothly homotopic maps between manifolds induce identical maps on the de Rham cohomology~\cite[Prop. 17.10]{Lee2013}. We sketch the relationship between this fact and Poincar\'e operators. For smooth manifolds $M,N$, and $I=[0,1]$ and smoothly homotopic maps $F\simeq G: M \to N$, let $i_t: M \to M\times I$ be $i_t(p) = (p,t)$ for all $p\in M, t \in I$ and $\Phi:M\times I \to N$ the smooth homotopy between $F$ and $G$ with $\Phi \circ i_0 = F$ and $\Phi \circ i_1 = G$. The usual proof for induced equality on cohomology is first to show existence of $h:\Omega^k(M\times I) \to \Omega^{k-1}(M)$ satisfying $dh + hd = i_1^* - i_0^*$. Let $S_{(q,s)} = \left(0, \left.\frac{\partial}{\partial s}\right|_s\right)$ be the vector field on $M \times I$ where the identification $T_{(q,s)}M \simeq T_q M \times T_s I$ has been made. Then an explicit formula for $h$ is
$h \omega := \int_0^1i_t^*(S\hook\omega) dt$, 
for $\omega \in \Omega^k(M\times I)$~\cite[Lemma 17.9]{Lee2013}. The map $K = h \circ \Phi^*: \Omega^k(N) \to \Omega^{k-1}(M)$ is the homotopy operator satisfying $dK + Kd = G^* - F^*$. For $M=N$ smoothly contractible, $F=c_a:M\to M$ the constant map to a fixed $a \in M$ and $G = \id_M$ the homotopy operator is called the Poincar\'e operator $\mathfrak{p}_k : \Omega^k(M) \to \Omega^{k-1}(M)$. Thus for $\alpha \in \Omega^k(M)$
\begin{equation}\label{eq:smooth-poincare}
    \mathfrak{p}_k \alpha = \int_0^1 i_t^*\left( S\hook (\Phi^* \alpha)\right) \, dt \, .
\end{equation}

\subsection{Singular simplices and chains}
We next briefly introduce the theory of singular simplices and chains as it is useful for our construction and less familiar in the numerical analysis literature. Much of the content in this subsection is adapted from~\cite[Chapter 4]{Munkres1984} and~\cite[Chapter 18]{Lee2013}.

We denote by $Q_k$, the standard (reference) $k$-simplex in $\R^n$. This is the $k$-simplex with vertices $q_0, \dots, q_k$ where the points $q_i\in \R^n$ are such that
\begin{equation*}
    q_0 = (0, 0, \dots, 0), \:
    q_1 = (1, 0, \dots, 0), \:
    \ldots, \:
    q_n =(0, 0, \dots, 1).
\end{equation*}

\begin{definition}[Singular simplices]
     Given a topological space $M$, a singular $k$-simplex is any mapping $T:Q_k \to M$. We let $\Xi_k(M)$ denote the set of all singular $k$-simplices which map into $M$. 
\end{definition}

\begin{definition}[Singular chains]
    The space of singular $k$-chains on $M$ is the (infinite dimensional) vector space $\Ss_k(M) = \left\{ \sum_{T\in \Xi_k} c_T T : c_T = 0 \text{ for all but finitely many }T\in \Xi_k(M)\right\}$.
\end{definition}

Of particular note are the linear singular simplices which have direct correspondence to the geometric simplices introduced in the previous subsection.
 
\begin{definition}[Linear singular simplices]
 For an arbitrary ordered set of points $\{a_0, \dots, a_k\}$ (not necessarily affinely independent) in $\R^n$, there is a unique affine map $L[a_0, \dots, a_k]:Q_k \to [a_0, \dots, a_k]$ which maps $q_i$ to $a_i$ for all $i\in\{0, \dots, k\}$. This map $L[a_0, \dots, a_k]$ has the explicit representation 
 \begin{equation}\label{eq:lrep}
     L[a_0, \dots, a_k](x_1, \dots, x_k) = a_0 + \sum_{i=1}^k x_i (a_i-a_0).
 \end{equation}
Such a map is called the linear singular simplex determined by $\{a_0, \dots, a_k\}$. In particular, one can identify any $\sigma\in X$ with $L\sigma$. The linear singular simplex is oriented with the orientation given by the ordering of the corresponding points. We let $\Xi_k^{\lin}(M)$ denote the set of linear singular $k$-simplices associated to points in the topological space $M$.
\end{definition}

We will also need to consider singular simplices with Lipschitz regularity. Such singular simplices are sufficient for defining integration and in particular for Stokes' theorem to hold \cite[Section 8]{Whitney1952}.

\begin{definition}[Lipschitz singular simplices]
    A Lipschitz singular $k$-simplex of a topological space $M$ is a Lipschitz map $T:Q_k \to M$. We denote the set of Lipschitz singular $k$-simplices of $M$ by $\Xi_k^{\sm}(M)$.
\end{definition}

We denote by $\Ss_k^{\lin}(M)$ and $\Ss_k^{\sm}(M)$ the respective spaces of linear and Lipschitz singular $k$-chains on $M$. These are the subspaces of $\Ss_k(M)$ which are respectively finite formal sums of linear singular $k$-simplices and finite formal sums of Lipschitz singular $k$-simplices. When the domain is clear, we may suppress the denotation of $M$ (e.g. write $\Xi_k$ instead of $\Xi_k(M)$).

Now we define the boundary operator $\partial_k^\Ss: \Ss_k\to\Ss_{k-1}$ by
\begin{equation}\label{def:partialSs}
    \partial_k^\Ss T = \sum_{i=0}^k (-1)^i T \circ L[q_0, \dots, \widehat{q_i}, \dots, q_k],
\end{equation}
where $L[q_0, \dots, \widehat{q_i}, \dots, q_k]$ is the linear singular simplex mapping from $Q_{k-1}$ to $Q_k$. The definition is then extended to all of $\Ss_k$ by linearity. One can check that with these definitions, $\partial_k^\Ss \partial_{k+1}^\Ss = 0$ and we have that $(\Ss_\bullet, \partial_\bullet^\Ss)$ is a complex of singular chains.

We have a simpler formula for the boundary operator on the linear singular $k$-chains as
\begin{equation}\label{def:partialSslin}
    \partial_k^\Ss L[a_0, \dots, a_k] = \sum_{i=0}^k (-1)^i L[a_0, \dots, \widehat{a_i}, \dots, a_k],
\end{equation}
for any set of $k$ points $ a_0, \dots, a_k $ in $\R^n$. Hence, the linear singular chains form a complex with the boundary operator $\partial_k^{\Ss}$. Moreover, since composition of a Lipschitz map with a linear map yields a Lipschitz map, the boundary operator $\partial_k^{\Ss}$ maps $\Ss_k^{\sm}$ to $\Ss_{k-1}^{\sm}$ and so the Lipschitz singular chains also form a complex with this boundary operator. 

We also have the Stokes' theorem on Lipschitz singular chains. Before noting this result, we first recall the definition of integration over Lipschitz singular chains. The integral of $\omega$ over a Lipschitz singular simplex, $T$, is defined as
\begin{equation}\label{def:intSing}
    \int_{T} \omega := \int_{Q_k} T^*\omega,
\end{equation}
where $T^*\omega$ is the pullback of $\omega$ by $T$. Then the definition is extended to $\Ss_k^{\sm}$ by linearity.

\begin{remark}[Degenerate simplices]\label{rem:deg}
    If one integrates over the linear singular simplex $T= L[a_0, \dots, a_k]$ when the points $a_0, \dots a_k$ are not affinely independent, the Jacobian $DL[a_0, \dots a_k]$ is rank deficient, and so the integral in \eqref{def:intSing} is zero.
\end{remark}

The Stokes' theorem for Lipschitz singular chains states that for a smooth enough $(k-1)$-form $\omega$ and a Lipschitz singular $k$-chain $C\in \Ss_k^{\sm}$, it holds that
\begin{equation}\label{eq:StokesSing}
    \int_C d^{k-1} \omega = \int_{\partial_k^\Ss C} \omega.
\end{equation}
The above statement for smooth chains can be found in \cite[Theorem 18.12]{Lee2013} and for Lipschitz chains, one can see \cite[Section 8]{Whitney1952}.

\subsection{Simplicial and singular chain maps}

\begin{definition}[Simplicial map]
    Let $X$ and $Y$ be simplicial complexes. Let $\widetilde{f}: X^0 \to Y^0$ be a map between their vertex sets such that for every simplex $[v_0, \ldots, v_k]$ of $X$, the set $[\widetilde{f}(v_0), \ldots, \widetilde{f}(v_k)]$ is a simplex of $Y$. Then, $\widetilde{f}$ can be extended to a continuous map, $f:|X|\to |Y|$ which maps each simplex of $X$ linearly onto a simplex of $Y$. Such $f$ is called a \emph{simplicial map}. As is standard, we may also write $f:X\to Y$ is a simplicial map when we are only interested in the action of $f$ on the simplices of $X$.
\end{definition}

\begin{definition}[Induced simplicial chain map]
    Let $X$ and $Y$ be simplicial complexes and $f:X\to Y$ be a simplicial map. The \emph{induced simplicial chain map} is a family of homomorphisms $\{f_\#^k\}$ where $
f_\#^k : \C_k(X) \longrightarrow \C_k(Y)$, is
defined on oriented simplices by
\[
f_\#^k ([v_0, \ldots, v_k]) =
\begin{cases}
  [f(v_0), \ldots, f(v_k)], & \text{if } f(v_0), \ldots, f(v_k) \text{ are distinct},\\[4pt]
  0, & \text{otherwise.}
\end{cases}
\]
This definition extends linearly to all chains. 
\end{definition}

We note the following two basic properties about induced simplicial chain maps which we will use later. First, from \cite[Lemma 12.1]{Munkres1984}, it holds that for any simplicial map $f:X\to Y$ and oriented $k$-simplex $\sigma\in X^k$, 
\begin{equation}\label{eq:simplicial_boundary_sharp_commute}
    \partial_k f_\#^k(\sigma) = f_\#^{k-1}(\partial_k \sigma).
\end{equation}
Also, as noted in the proof of \cite[Theorem 12.2]{Munkres1984}, it holds that for simplicial maps $f: X\to Y$ and $g:Y\to Z$ where $X,Y,Z$ are simplicial complexes,
\begin{equation}\label{eq:simplicial_map_functoriality}
    (g \circ f)_\# = g_\# \circ f_\#.
\end{equation}

We have an analogous concept to the simplicial chain map for singular chains.
\begin{definition}[Induced singular chain map]
    Let $M$ and $N$ be topological spaces and $f:M \to N$ be a continuous map. The \emph{induced singular chain map} is a family of homomorphisms $\{f_k^\#\}$ where  $f_k^\# :\Ss_k(M) \to \Ss_k(N)$, is defined on singular simplices by $f_k^\#(T) = f \circ T$ for any $T\in \Xi_k(M)$.
\end{definition}

The induced singular chain map commutes with the singular boundary operator \cite[Theorem 29.1]{Munkres1984}. That is, for any continuous map $f:M\to N$ and singular $k$-simplex $T\in \Xi_k(M)$,
\begin{equation}
    \partial_k^\Ss f_k^\#(T) = f_{k-1}^\#(\partial_k^\Ss T).
\end{equation}

\begin{remark}
    If $X$ and $Y$ are simplicial complexes and $f: |X| \to |Y|$ is a simplicial map, then 
    \begin{equation}\label{eq:singular_boundary_sharp_commute}
        f^\#(L\sigma) = f_\#(\sigma) \quad \forall \sigma\in X.
    \end{equation}
\end{remark}

\section{Discrete Poincar\'e operators}\label{sec:discPoincidea}

We now introduce the discrete Poincar\'e operators as duals in the appropriate sense to simplicial and singular chain-valued ``generalized cone operators.'' This terminology is taken from \cite{Desbrun2005}. Such operators are standard constructions in simplicial topology (see e.g. \cite[Section 8]{Munkres1984} for a simplicial cone operator and \cite[Section 29]{Munkres1984} for a singular cone operator). However, in the topology literature the definitions are not always suitable for computation (e.g. they may not even be constructive), in contrast to the operators we present in this article.

\begin{remark}\label{rem:DEC-hook}
    While many previous constructions of discrete Poincar\'e operators seem to be based on a cone operator framework, it is not the only path to such operators. An alternate approach uses a DEC discretization of the smooth construction in~\S\ref{subsec:smooth-poincare}. The construction builds on the idea that the interior product (hook or contraction operator) and pullback of forms have  DEC discretizations. These are the two main ingredients in the integrand $i_t^*\left( S\hook (\Phi^* \alpha)\right)$ of~\eqref{eq:smooth-poincare}. Let $M$ be a smooth manifold and $\Gamma$ a dimension $k$ submanifold. Let $X$ be a smooth vector field on $M$ that is transverse to $\Gamma$. Under appropriate assumptions, the advection of $\Gamma$ by the flow of $X$ for a short time $t$ creates the \emph{extrusion} of $\Gamma$, a submanifold $E_t^X(\Gamma)$ of dimension $k+1$. Then $\left.\frac{d}{dt}\right|_{t=0} \int_{E_t^X(\Gamma)} \omega = \int_\Gamma X\hook \omega$ for $\omega \in \Omega^{k+1}(M)$. This is the basis for one of the discretizations of the interior product in DEC~\cite{Hirani2003} and in FEEC~\cite{Bossavit2003,HeHi2011}. While this idea inspires our combinatorial construction, we present it packaged into a cone operator construction for uniformity of presentation. 
\end{remark}

\subsection{Generalized cone operators} 
\begin{definition}[Simplicial cone operator]
    Given a simplicial complex $X$, a \emph{simplicial cone operator} is any homomorphism $\Co_k:\C_k(X) \to \C_{k+1}(X)$ such that for any $\sigma\in X^k$, 
    \begin{subequations}\label{eq:homsimp}
    \begin{alignat}{2}
        \Co_{k-1}(\partial_k \sigma)+ \partial_{k+1} \Co_k(\sigma) &= \sigma \qquad &&\forall \sigma \in X^k, k\geq 1 \label{eq:hom1simp}\\
        \partial_1 \Co_0(\sigma) &= \sigma - a \qquad &&\forall \sigma \in X^0, \label{eq:hom2simp}
    \end{alignat}
    \end{subequations}
    where $a\in X^0$ is a given vertex. We refer to $a$ as the contraction vertex.
\end{definition}

\begin{definition}[Singular cone operator]
    Given a simplicial complex $X$, a \emph{singular cone operator} is any operator $\Co_k^\Ss:\C_k(X) \to \Ss_{k+1}^{\sm}(X)$ such that 
    \begin{subequations}\label{eq:homsing}
    \begin{alignat}{2}
        \Co_{k-1}^\Ss(\partial_k \sigma)+ \partial_{k+1}^\Ss \Co_k^\Ss(\sigma) &= L\sigma \qquad &&\forall \sigma \in X^k, k\geq 1 \label{eq:hom1sing}\\
        \partial_1^\Ss \Co_0^\Ss(\sigma) &= L\sigma - L[a] \qquad &&\forall \sigma \in X^0. \label{eq:hom2sing}
    \end{alignat}
    \end{subequations}
    where $a$ is a given point in $|X|$. We refer to $a$ as the contraction point.
\end{definition}

\subsection{Discrete Poincar\'e operators}
\begin{definition}[Combinatorial discrete Poincar\'e operator]\label{def:comb_disc_Poinc}
    Let $X$ be a simplicial complex and $\Co_k$ be a simplicial cone operator. Then, we define the associated discrete Poincar\'e operator, $P^k:\C^k(X) \to \C^{k-1}(X)$ for $k\ge 1$, as
    \begin{equation}\label{eq:defcombdiscPoinc}
        \bl P^k \alpha, \sigma \br = \bl \alpha,\, \Co_{k-1}(\sigma) \br,
    \end{equation}
    for any $\alpha\in\C^k(X)$ and $\sigma\in X^{k-1}$.
\end{definition}

\begin{proposition}
Let $X$ be a simplicial complex and $\Co_k$ be a simplicial cone operator with contraction vertex $a$. Let the associated discrete Poincar\'e operator of Definition \ref{def:comb_disc_Poinc} be denoted $P^k$.
Then, it holds
\begin{subequations}\label{eq:homotopy_formula_dPDEC}
    \begin{alignat}{2}
    (\dd^{k-1} P^k + P^{k+1} \dd^{k}) \alpha &= \alpha, \qquad && \forall \alpha\in \C^k, k \geq 1,\label{eq:combdP+Pd}\\
    P^{k+1}\dd^k \alpha &= \alpha - \pi_{\C}^* \alpha, \qquad && \forall \alpha \in \C^0,,\label{eq:combPd}
\end{alignat}
\end{subequations}
where $\pi_{\C}^* \alpha$ is such that $\bl \pi_{\C}^* \alpha, v\br = \bl \alpha, a\br$ for any $v\in X^0$.
\end{proposition}
\begin{proof}
For $\alpha\in\C^k$ and $\sigma\in X^k$, $k\geq 1$ using the definition of the coboundary map, \eqref{eq:defcombdiscPoinc}, and \eqref{eq:hom1simp},
    \begin{alignat*}{2}
        \bl (\dd^{k-1} P^k + P^{k+1} \dd^{k}) \alpha, \sigma\br &= \bl P^k \alpha, \partial_k \sigma\br + \bl\dd^k \alpha, \Co_{k}(\sigma)\br\\
        &= \bl \alpha, \Co_{k-1}(\partial_k \sigma)+ \partial_{k+1} \Co_k(\sigma)\br = \bl \alpha, \sigma\br.
    \end{alignat*}
For $\alpha\in \C^{0}$ and $v\in X^0$ we have $\bl P^1 \dd^0 \alpha, v\br = \bl \dd^0 \alpha, \Co_0(v)\br = \bl \alpha, \partial_{1}\Co_{0}(v) \br =  \bl \alpha, v\br - \bl\alpha,  a\br$.

\end{proof}

\begin{definition}[Whitney form based discrete Poincar\'e operator]\label{def:Whit_disc_Poinc}
    Let $X$ be a simplicial complex and $\Co_k^\Ss$ be a singular cone operator. Also, let $V_h^k$ be the space of Whitney $k$-forms over the simplicial complex $X$. Then, we define the associated discrete Poincar\'e operator over Whitney forms, $\Po^k:V_h^k \to V_h^{k-1}$ by setting
    \begin{equation}\label{eq:Pdef}
        \int_\sigma \Po^k u = \int_{\Co_{k-1}^{\Ss}(\sigma)} u, \quad \forall \sigma\in X^{k-1}
    \end{equation}
    for any $u\in V_h^k$.
\end{definition}

\begin{proposition}\label{thm:poincareWhit}
Let $X$ be a simplicial complex and $\Co_k^\Ss$ be a singular cone operator with contraction point $a$. Let the associated discrete Poincar\'e operator of Definition \ref{def:Whit_disc_Poinc} be denoted $\mathcal{P}^k$.
Then, it holds
\begin{subequations}
    \begin{alignat}{2}
        (d^{k-1} \Po^k + \Po^{k+1} d^k) u &= u, \quad && \forall u\in V_h^k, k\geq 1 \label{eq:hom1Wgeneral}\\
        \Po^{k+1} d^k u &= u - u(a) \quad  && \forall u\in V_h^0 \label{eq:hom2Wgeneral}
    \end{alignat}
\end{subequations}
\end{proposition}

\begin{proof}
Let $u \in V_h^k$, $k\geq 1$. Then we must show that 
\begin{equation}
\int_{\tau} d^{k-1} \Po^k u + \int_{\tau} \Po^{k+1} d^k u  =\int_{\tau} u \quad \forall \tau \in X^k.
\end{equation} 

Hence, we compute 
\begin{alignat*}{2}
\int_{\tau} d^{k-1} \Po^k u  + \int_{\tau} \Po^{k+1} d^k u  &= \int_{\partial_k \tau} \Po^{k} u + \int_{\Co_k^\Ss(\tau)} d^k u \qquad && \text{by \eqref{eq:Pdef}, \eqref{eq:StokesSimp}}  \\
&= \int_{\Co_{k-1}^\Ss(\partial_k \tau)} u + \int_{\partial_{k+1}^\Ss \Co_k^\Ss(\tau)} u \qquad && \text{by \eqref{eq:StokesSing}, \eqref{eq:Pdef}} \\
&= \int_{L\tau} u = \int_{\tau} u \qquad && \text{by \eqref{eq:hom1sing}}.
\end{alignat*}
The last equality is from the definition of the integral of a $k$-form over a $k$-simplex. On the other hand, for $u\in V_h^0$ and $v\in X^0$, we have
\begin{equation*}
    (\Po^{1} d^0 u) (v) = \int_{\Co_0^\Ss(v)} d^0 u = \int_{\partial_1^\Ss \Co_0^\Ss(v)} u = \int_{L[v]-L[a]} u = u(v) - u(a),
\end{equation*}
where in the second to last step we used \eqref{eq:hom2sing}.
\end{proof}

Using the Whitney and de Rham maps we can adapt the operator to cochains. 
\begin{definition}[Whitney form based discrete Poincar\'e operator on cochains]\label{Whitney_cochain_op}
Suppose we are in the setting of the previous proposition. Then we define a discrete Poincar\'e operator $\Poo^k: \mathcal{C}^k \rightarrow \mathcal{C}^{k-1}$ by
\begin{equation}
 \Poo^k := \dr^{k-1} \Po^k W^k.   
\end{equation}
\end{definition}

\begin{corollary}\label{cor:Whitney_cochain_op}
It holds, 
\begin{subequations}
    \begin{alignat}{2}
          (\dd^{k-1} \Poo^{k}+\Poo^{k+1} \dd^k)\alpha &= \alpha, \quad && \forall \alpha\in \C^k, k\geq 1 \label{hom1C}\\
        \Poo^{k+1} \dd^k \alpha &= \alpha - \pi_W \alpha \quad  && \forall \alpha\in \C^0 \label{hom2C}
    \end{alignat}
\end{subequations}
where $\pi_W \alpha \in \C^0$ is such that $\bl \pi_W \alpha, v \br = (W^0\alpha) (a)$ for all $v \in X^0$.
\end{corollary}
\begin{proof}
  For $\alpha\in\C^k$, $k\geq 1$, we have 
\begin{alignat*}{2}
  (\dd^{k-1} \Poo^{k}+\Poo^{k+1} \dd^k)\alpha &= \dd^{k-1} \dr^{k-1} \Po^k W^k\alpha + \dr^{k} \Po^{k+1} W^{k+1} \dd^k\alpha \qquad &&   \\
  &= \dr^k (d^{k-1} \Po^k+\Po^{k+1} d^k) W^k\alpha \qquad && \text{by \eqref{aux1}, \eqref{aux2}}   \\
  &= \dr^k W^k\alpha = \alpha  \qquad && \text{by } \eqref{eq:hom1Wgeneral}, \eqref{aux3}\\
\end{alignat*}
For $\alpha\in\C^0$, using \eqref{aux1} and \eqref{eq:hom2Wgeneral}, we have
\begin{equation*}
    \Poo^{1} \dd^0 \alpha = \dr^0 \Po^{1} W^{1} \dd^0\alpha = \dr^0 \Po^1 d^0 W^0\alpha = \dr^0 \left(W^0 \alpha - (W^0\alpha)(a)\right) = \alpha - \pi_W \alpha.
\end{equation*}
\end{proof}

\begin{remark}
    The continuous Poincar\'e operator satisfies the complex property $P\circ P = 0$. We do not verify this property for the discrete operators we construct in this paper. However, from our constructions, one can easily obtain discrete Poincar\'e operators which satisfy the complex property following \cite[Theorem 5]{CapHu2023}. That is, given a discrete Poincar\'e operator, $P$ which satisfies \eqref{eq:homotopy_formula}, the operator $\widetilde{P} := P - dPP$ will satisfy the complex property along with \eqref{eq:homotopy_formula}.

\end{remark}

In the remainder of this paper, we construct simplicial and singular generalized cone operators under various assumptions on either the simplicial complex or underlying domain. We first give some simple constructions in Section \ref{sec:cones_special_cases}, before providing a more general construction in Section \ref{sec:cones_general}.

\section{Some simple examples of cone operators}\label{sec:cones_special_cases}

\subsection{Simplicial cone operator on collapsible simplicial complexes}\label{sec:collapse_cone}
For ``collapsible'' simplicial complexes a simplicial cone operator 
is developed in~\cite[Section 2.3]{Desbrun2005} and~\cite{Pitassi2022} in independent contexts, and we present it here for completeness. This construction in fact follows  from a theorem of Whitehead \cite{Whitehead1950} as we briefly detail in Appendix \ref{app:Whitehead}.

We first remark on the contexts in which the simplicial cone operator we will present has previously appeared. In \cite{Desbrun2005}, the authors use a sequence of ``one ring cone augmentations'' to define the simplicial cone operator. The set of simplicial complexes which can be grown from these augmentations are called ``generalized star-shaped complexes'' in \cite{Desbrun2005}. This notion is actually equivalent to the notion of strong collapsibility of \cite{Barmak2012} which is strictly stronger than collapsibility. However, their construction works just as well in the collapsible setting, a fact which we detail in this section.

In~\cite{Pitassi2022}, the authors construct a discrete vector potential algorithm on collapsible simplicial complexes. The authors utilize a so-called ``complete acyclic matching" to invert the discrete curl operator on $1$-cochains (i.e. the coboundary operator $\dd^1$) over $3$-dimensional polyhedral cell complexes. In particular, they establish that the existence of a complete acyclic matching of $1$-chains is equivalent to the complex being collapsible. In fact, they demonstrate that a collapse sequence directly yields such a complete acyclic matching. When this specific matching is utilized, the back substitution process in \cite[Algorithm 1]{Pitassi2022} is algebraically identical to the discrete Poincar\'e operator obtained from the simplicial cone operator introduced below. We also note that the authors provide an algorithm (\cite[Algorithm 3]{Pitassi2022}) which can handle non-collapsible simplicial complexes.

\begin{definition}[Elementary simplicial collapse]\label{def:elementary_collapse}
    Let $X$ be a simplicial complex. Let $\sigma$ and $\tau$ be simplices of $X$ such that
\begin{enumerate}
    \item $\tau$ is a face of $\sigma$ of codimension 1 (i.e., $\dim \tau = \dim \sigma - 1$).
    \item $\sigma$ is the only simplex in $X$ that has $\tau$ as a face (other than $\tau$ itself).
\end{enumerate}
The operation of going from the complex $X$ to the complex $X' := X \setminus \{\sigma, \tau\}$, denoted $X\searrow X'$ is called an \emph{elementary simplicial collapse}.  The simplex $\tau$ is called a \textit{free face}. The operation of going from $X'$ to $X$ is called an \emph{elementary simplicial expansion}.
\end{definition}

While the collapse $X\searrow X' = X \setminus \{\sigma, \tau\}$ removes $\sigma$ and $\tau$ from the collection of closed simplices, in terms of the underlying topological space $|X|$, this operation corresponds to removing the open simplices (interiors) $\Int(\sigma)$ and $\Int(\tau)$, such that $|X'| = |X| \setminus (\Int(\sigma) \cup \Int(\tau))$.

\begin{definition}[Collapsible simplicial complex]
    A simplicial complex $X$ is called \emph{collapsible} if there exists a finite sequence of elementary simplicial collapses $X= X_m \searrow X_{m-1} \searrow \cdots \searrow X_0 = a$ that reduces $X$ to a single vertex $a\in X^0$. 
\end{definition}

An illustration of a collapse sequence is given in Figure \ref{fig:collapse_sequence}.

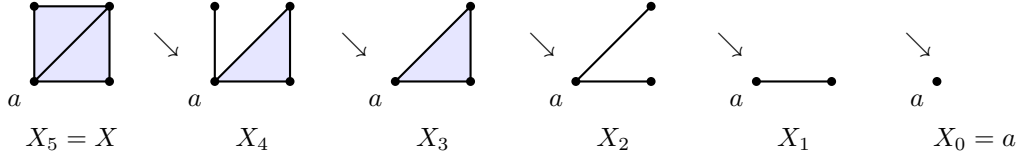
\begin{figure}[htpb]
    \centering
    \begin{tikzpicture}[scale=1.0]
        \tikzset{
            vertex/.style={circle, fill=black, inner sep=1.2pt},
            edge/.style={thick, black},
            face/.style={fill=blue!10}
        }
        
        \begin{scope}[xshift=0cm]
            \fill[face] (0,0) -- (1,0) -- (1,1) -- cycle;
            \fill[face] (0,0) -- (1,1) -- (0,1) -- cycle;
            \draw[edge] (0,0) -- (1,0) -- (1,1) -- (0,1) -- cycle;
            \draw[edge] (0,0) -- (1,1);
            \node[vertex] at (0,1) {};
            \node[vertex, label=below left:$a$] at (0,0) {};
            \node[vertex] at (1,0) {};
            \node[vertex] at (1,1) {};
            \node[below] at (0.5,-0.5) {$X_5=X$};
        \end{scope}

        \node at (1.75, 0.5) {$\searrow$};

        \begin{scope}[xshift=2.4cm]
            \fill[face] (0,0) -- (1,0) -- (1,1) -- cycle;
            \draw[edge] (0,0) -- (1,0) -- (1,1);
            \draw[edge] (0,0) -- (1,1);
            \draw[edge] (0,0) -- (0,1);
            \node[vertex] at (0,1) {};
            \node[vertex, label=below left:$a$] at (0,0) {};
            \node[vertex] at (1,0) {};
            \node[vertex] at (1,1) {};
            \node[below] at (0.5,-0.5) {$X_4$};
        \end{scope}

        \node at (4.25, 0.5) {$\searrow$};

        \begin{scope}[xshift=4.8cm]
            \fill[face] (0,0) -- (1,0) -- (1,1) -- cycle;
            \draw[edge] (0,0) -- (1,0) -- (1,1);
            \draw[edge] (0,0) -- (1,1);
            \node[vertex, label=below left:$a$] at (0,0) {};
            \node[vertex] at (1,0) {};
            \node[vertex] at (1,1) {};
            \node[below] at (0.5,-0.5) {$X_3$};
        \end{scope}

        \node at (6.75, 0.5) {$\searrow$};

        \begin{scope}[xshift=7.2cm]
            \draw[edge] (0,0) -- (1,0);
            \draw[edge] (0,0) -- (1,1);
            \node[vertex, label=below left:$a$] at (0,0) {};
            \node[vertex] at (1,0) {};
            \node[vertex] at (1,1) {};
            \node[below] at (0.5,-0.5) {$X_2$};
        \end{scope}

        \node at (9.25, 0.5) {$\searrow$};

        \begin{scope}[xshift=9.6cm]
            \draw[edge] (0,0) -- (1,0);
            \node[vertex, label=below left:$a$] at (0,0) {};
            \node[vertex] at (1,0) {};
            \node[below] at (0.5,-0.5) {$X_1$};
        \end{scope}

        \node at (11.75, 0.5) {$\searrow$};

        \begin{scope}[xshift=12.0cm]
            \node[vertex, label=below left:$a$] at (0,0) {};
            \node[below] at (0.5,-0.5) {$X_0=a$};
        \end{scope}

    \end{tikzpicture}
    \caption{A sequence of five elementary simplicial collapses reducing a complex of two adjacent triangles to a single vertex $a$.}
    \label{fig:collapse_sequence}
\end{figure}

We now suppose $X$ is a collapsible simplicial complex with collapse sequence $X= X_m \searrow X_{m-1} \searrow \cdots \searrow X_0 = a$ for $a\in X^0$. We will consider the sequence of elementary simplicial expansions obtained from reversing the collapse sequence: $a = X_0 \nearrow X_{1} \nearrow \cdots \nearrow X_m = X$. Rather than defining the simplicial cone operator $\Co_k$ for each $k$ separately, we will instead define the operator inductively using this expansion sequence.

\begin{definition}[Cone operator from simplicial collapse]\label{def:collapse_cone}
    Let $\C(X):=\oplus_{k\geq 0}\C_k(X)$. We will define $\Co:\C(X)\to\C(X)$ and obtain $\Co_k$ by restricting to specific $k$-values. First, we define
\begin{equation}\label{eq:collapse_cone_def_a}
    \Co(a) := 0,
\end{equation}
where $0$ is the zero chain. Thus, we have defined $\Co$ on $\C(X_0)$. 

Suppose we have defined $\Co$ on all of $\C(X_{j-1})$ for some $j\in\{0,\dots, m\}$. We now define $\Co$ on $\C(X_{j})$ by defining its action on the two simplices that are added when performing the simplicial expansion $X_{j-1}\nearrow X_{j}$. Let these two simplices be denoted $\sigma_{j}$ and $\tau_{j}$ where $\tau_{j}$ is the free face. That is, $X_j = X_{j-1}\cup\{\sigma, \tau\}$. We then define
\begin{subequations}\label{eq:collapse_cone_def}
    \begin{alignat}{2}
        \Co(\sigma_j) &:= 0\label{eq:collapse_cone_def_1}\\
        \Co(\tau_j) &:= \sigma_j - \Co(\partial \sigma_j - \tau_j),\label{eq:collapse_cone_def_2}
    \end{alignat}
\end{subequations}
where $\partial = \partial_k$ when $\sigma_j$ is a $k$-simplex.
\end{definition}

\begin{proposition}\label{prop:collapse_cone}
    The cone operator of Definition \ref{def:collapse_cone} is a simplicial cone operator.
\end{proposition}

Before we prove Proposition \ref{prop:collapse_cone}, we first note the following lemma.
\begin{lemma}\label{lem:aux_collapse_cone}
    Let $\sigma_j$ and $\tau_j$ be the two simplices added to $X_{j-1}$ in the simplicial expansion $X_{j-1}\nearrow X_j$. Let $k$ be the dimension of $\sigma_j$. Then, it holds that $\partial\sigma_j - \tau_j \in \C_{k-1}(X_{j-1})$.
\end{lemma}

\begin{proof}
    Because $X_j$ is a simplicial complex, $\partial \sigma_j$ is a chain of $(k-1)$-simplices in $X_j$. That is, $\partial \sigma_j = \sum_{i=0}^k \eta_i$ where $\eta_i\in X_j^{(k-1)}$ for all $i\in\{0,\dots, k\}$. Moreover, by Definition \ref{def:elementary_collapse}, $\tau_j$ is a $(k-1)$-face of $\sigma_j$ and so there exists $i'\in \{0,\dots, k\}$ such that $\eta_{i'} = \tau_j$. Finally, since $X_{j-1}=X_j \setminus \{\sigma_j,\tau_j\}$, it must be that $\eta_i\in X_{j-1}^{(k-1)}$ for all $i\neq i'$. Thus, $$\partial \sigma_j - \tau_j = \sum_{i\in \{0,\dots,k\}\setminus i'} \eta_i \in \C_{k-1}(X_{j-1}).$$
\end{proof}
We are now ready to prove Proposition \ref{prop:collapse_cone}.
\begin{proof}[Proof of Proposition \ref{prop:collapse_cone}]
    We prove this by induction on the simplicial expansion sequence. We first note that by \eqref{eq:collapse_cone_def_a}, we immediately have \eqref{eq:homsimp} for the vertex $a$. Now, suppose that \eqref{eq:homsimp} holds for all simplices in $X_{j-1}$. We will show that \eqref{eq:homsimp} holds for all simplices in $X_{j}$. In particular, we simply need to show that \eqref{eq:homsimp} holds for the two simplices $\sigma_{j}$ and $\tau_j$ which are added in the simplicial expansion $X_{j-1}\nearrow X_j$. 
    We first compute using \eqref{eq:collapse_cone_def_2} and the fact that \eqref{eq:homsimp} holds on $\partial \sigma_j - \tau_j$ (since $\partial \sigma_j - \tau_j\in \C(X_{j-1})$ due to Lemma \ref{lem:aux_collapse_cone})
    \begin{align*}
        \Co(\partial \tau_j) + \partial \Co(\tau_j) &= \Co(\partial \tau_j) + \partial \sigma_j - \partial \Co(\partial \sigma_j - \tau_j)\\
        &= \Co(\partial \tau_j) + \partial \sigma_j + \Co(\partial (\partial \sigma_j - \tau_j))-(\partial \sigma_j - \tau_j) = \tau_j.
    \end{align*}
    Then, using \eqref{eq:collapse_cone_def_1} and \eqref{eq:collapse_cone_def_2}, we have
    \begin{align*}
        \Co(\partial \sigma_j) + \partial \Co(\sigma_j) &= \Co(\partial \sigma_j -\tau_j) + \Co(\tau_j) + \partial \Co(\sigma_j)= \sigma_j - \Co(\tau_j) + \Co(\tau_j) = \sigma_j.
    \end{align*}
\end{proof}

\subsection{Singular cone operator on star-shaped domains}\label{sec:discPoincWhitsimple}

\begin{definition}[Star-shaped domain]\label{def:star_domain}
    We say a domain $\Omega$ is \emph{star-shaped with respect to a point $a\in \Omega$} if for any $x\in\Omega$, the line segment connecting $x$ and $a$ lies completely in $\Omega$.
\end{definition}
 We can now give the simple formula for the singular cone operator in this setting.
\begin{definition}[Singular cone operator on star-shaped domains]\label{def:sing_cone_star}
    Given a simplicial complex $X$ with geometric realization $\Omega = |X|$ which is star-shaped with respect to a point $a$, we define the operator $\Co_k:\C_k(X)\to \Ss_{k+1}^{\lin}$ by setting
    \begin{equation}
        \Co_k^\Ss(\sigma) = L[a,\sigma], \quad \forall \sigma\in X^k,
    \end{equation}
    where we are using the notation that for $\sigma = [v_0, \dots, v_k] \in X^k$, $[a, \sigma] = [a, v_0, \dots, v_k]$.
\end{definition}

\begin{proposition}
    The operator defined in Definition \ref{def:sing_cone_star} is a singular cone operator.
\end{proposition}

\begin{proof}
    First, notice that the operator maps to the space of linear singular $(k+1)$-simplices which is a subspace of $\Ss_{k+1}^{\sm}$. Now, let $k\geq 1$ and $\sigma=[v_0, \ldots, v_k]$. Then
\begin{align*}
     \partial_{k+1}^\Ss \Co_k^\Ss(\sigma)&= \partial_{k+1}^\Ss L[a, v_0, v_1, \ldots, v_k]\\
     &= L\sigma + \sum_{i=0}^k (-1)^{i+1} L[a, v_0, \ldots \hat{v_i}, \ldots, v_k] = L\sigma - \Co_{k-1}^\Ss(\partial_k \sigma).
\end{align*}

On the other hand, for $\sigma\in X^0$, we have $\partial_{1}^\Ss \Co_0^\Ss(\sigma) = \partial_1^\Ss L[a, \sigma] = L\sigma - L[a]$.
\end{proof}

We remark that the operator in Definition \ref{def:sing_cone_star} is essentially the ``bracket operation'' of \cite[Section 29]{Munkres1984} up to orientation and identification of simplices with their singular counterparts. We now give some details to better elucidate the cone operator of Definition \ref{def:sing_cone_star} as well as the resulting discrete Poincar\'e operator.

We split $X^k$ into two disjoint sets for all $k$. We let $X_{\smid}^k$ be the set of all $\sigma \in X^k$ such that the point $a$ is coplanar to $\sigma$ and then define the set $X_{\rhd}^k = X^k\setminus X_{\smid}^k$. The image of $L[a,\sigma]$ will in general not be a simplex in $X^{k+1}$. However, it will be a geometric $(k+1)$-simplex lying in $\Omega$ if $\sigma\in X_{\rhd}^k$ as we are assuming $\Omega$ is star-shaped with respect to $a$. On the other hand, if $\sigma\in X_{\smid}^k$, then the image of $L[a,\sigma]$ is a degenerate $(k+1)$-simplex. Some illustrations are given in Figure \ref{fig:cones}.



\begin{figure}[h!]
\centering

\begin{subfigure}{0.44\textwidth}
    \centering
    \begin{tikzpicture}[
        scale=0.6,
        dot/.style={circle, fill, inner sep=1.5pt}
    ]
        \coordinate (C) at (0,0); 
        \foreach \angle/\i in {0/1, 60/2, 120/3, 180/4, 240/5, 300/6}{
            \coordinate (H\i) at (\angle:3cm);
        }
        \foreach \angle/\i in {0/1, 60/2, 120/3, 180/4, 240/5, 300/6}{
            \coordinate (P\i) at (\angle:1.5cm);
        }
        \coordinate (s1) at (H5);
        \coordinate (s2) at (P4);
        \coordinate (a) at (1.9,0.68571);


        \draw[dashed, black, thick] (a) -- (s1);
        \draw[dashed, black, thick] (a) -- (s2);
        
        \foreach \i in {1,...,6}{
            \pgfmathtruncatemacro{\j}{mod(\i, 6) + 1}
            \draw[gray, thin] (H\i) -- (P\i) -- (P\j) -- (H\j) -- cycle;
            \draw[gray, thin] (H\j) -- (P\i);
            \draw[gray, thin] (C) -- (P\i);
        }
        \draw[gray, thin] (P1)--(P2)--(P3)--(P4)--(P5)--(P6)--cycle;
        
        \draw[line width=1.5pt] (s1) -- (s2);
        \node at (-1.8, -1.2) {$\sigma$};
        \node[dot, label={[yshift=1pt, xshift=-2pt]right:$a$}] at (a) {};
    - \end{tikzpicture}
    \caption{$[a,\sigma]$ is a geometric $2$-simplex as $\sigma \in X_{\rhd}^{(1)}$.}
    \label{fig:hex_b_new_sigma}
\end{subfigure}
\hspace{4mm} 
\begin{subfigure}{0.44\textwidth}
    \centering
    \begin{tikzpicture}[
        scale=0.6,
        dot/.style={circle, fill, inner sep=1.5pt}
    ]
        \coordinate (C) at (0,0); 
        \foreach \angle/\i in {0/1, 60/2, 120/3, 180/4, 240/5, 300/6}{
            \coordinate (H\i) at (\angle:3cm);
        }
        \foreach \angle/\i in {0/1, 60/2, 120/3, 180/4, 240/5, 300/6}{
            \coordinate (P\i) at (\angle:1.5cm);
        }
        \coordinate (s1) at (P1);
        \coordinate (s2) at (P6);
        \coordinate (a) at (1.9,0.68571);


        \draw[dashed, black, thick] (a) -- (s1);
        \draw[dashed, black, thick] (a) -- (s2);
        
        \foreach \i in {1,...,6}{
            \pgfmathtruncatemacro{\j}{mod(\i, 6) + 1}
            \draw[gray, thin] (H\i) -- (P\i) -- (P\j) -- (H\j) -- cycle;
            \draw[gray, thin] (H\j) -- (P\i);
            \draw[gray, thin] (C) -- (P\i);
        }
        \draw[gray, thin] (P1)--(P2)--(P3)--(P4)--(P5)--(P6)--cycle;
        
        \draw[line width=1.5pt] (s1) -- (s2);
        \node at (1.45, -0.55) {$\sigma$};
        \node[dot, label={[yshift=1pt, xshift=-2pt]right:$a$}] at (a) {};
    - \end{tikzpicture}
    \caption{$[a,\sigma]$ is a degenerate 2-simplex as $\sigma\in X_{\smid}^{(1)}$}
    \label{fig:hex_b_sigma}
\end{subfigure}

\caption{Illustration of different scenarios for the region $[a,\sigma]$.}
\label{fig:cones}
\end{figure}
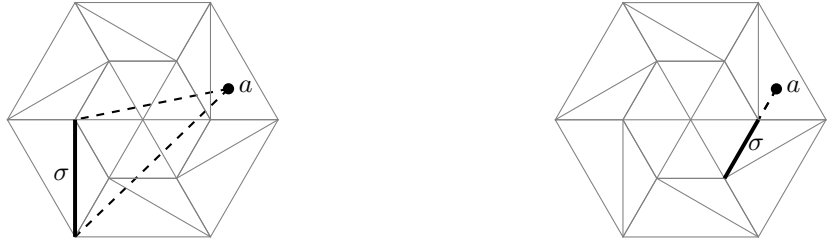


Following Definition \ref{def:Whit_disc_Poinc}, one obtains a discrete Poincar\'e operator by setting the corresponding degrees of freedom. In particular, for $u \in V_h^k$, the operator $\Po$ of Definition \ref{def:Whit_disc_Poinc} is such that 
\begin{equation}\label{discPoincWhitsimple}
\int_{\sigma} \tr_\sigma \Po u = \begin{cases}
    \int_{[a,\sigma]}  u \qquad &\forall \sigma \in X_{\rhd}^{(k-1)},  \\
0 \qquad &\forall \sigma \in X_{\smid}^{(k-1)}.
\end{cases}
\end{equation}

Note that the integration over the region $[a,\sigma]$ is well-defined precisely for $\sigma \in X_{\rhd}^{(k-1)}$. 

\begin{remark}\label{rem:Whit}
    In \cite[Chapter IV, Section 25]{Whitney1957}, it is shown that the continuous Poincar\'e operator, $\mathfrak{p}$, on star-shaped domains with the straight line contraction $\Phi(x,t) = (1-t)a + tx$ satisfies the property that for any simplex $\sigma\in\Omega$, and regular enough $k$-form $\omega$,
    \begin{equation}
        \int_\sigma \tr_\sigma \mathfrak{p} \omega = \int_{[a,\sigma]} \omega,
    \end{equation}
    if $a$ does not lie in the same plane as $\sigma$, and $\int_\sigma \mathfrak{p} \omega = 0$ otherwise.
    This implies that the discrete Poincar\'e operator defined in this section is in fact equivalent to the canonical projection of the continuous Poincar\'e operator. 
\end{remark}

In light of Remark \ref{rem:Whit}, we find the advantage of the discrete Poincar\'e operator in this section lies in its simplicity. In particular, the discrete operator can be directly computed if one can performs the appropriate integration in \eqref{discPoincWhitsimple} numerically. This involves numerically integrating piecewise polynomials over regions which cut through elements, which may be challenging, a subject which has previously been studied in \cite{Farrell2011} for instance. We provide our own numerical experiments for this operator in Section \ref{sec:numerical_experiments} below.

\section{Cone operators more generally}\label{sec:cones_general}

In this section, we outline general frameworks from which generalized cone operators can be defined. More explicit details pertaining to noted assumptions are given in the succeeding sections.

\subsection{Product Space Triangulations and Extrusions}
One can think about constructing discrete Poincar\'e operators by discretizing each ingredient in the definition of the smooth Poincar\'e operator in Section~\ref{subsec:smooth-poincare}. We take this approach to construct generalized cone operators which give the discrete Poincar\'e operators following the discussion in Section \ref{sec:discPoincidea}. In particular, as pointed out in Remark~\ref{rem:DEC-hook} the interior product (hook or contraction) operator will be discretized via a notion of extrusion which relies on triangulations of the product space $\Omega \times I$.

\begin{definition}[Admissible Product Complex]\label{def:apc}
Given a simplicial complex $X$ with geometric realization $\Omega = |X|$, a simplicial complex $K$ is called an \emph{admissible product complex} of $\Omega \times I$ if it satisfies the following three conditions:
\begin{enumerate}
    \item The domain underlying $K$ is exactly $\Omega\times I$: $|K| = \Omega \times I$
    \item For every simplex $\sigma \in X$, the set $\sigma \times I$ is the underlying space of a subcomplex of $K$ which we denote $K_\sigma$. 
    \item For every simplex $\sigma \in X$, the sets $\sigma \times \{0\}$ and $\sigma \times \{1\}$ are simplices of $K$. 
\end{enumerate}
\end{definition}

\begin{remark}
    In \cite[Lemma 19.1]{Munkres1984}, it is shown that an admissible product complex always exists for any simplicial complex $X$. The argument is constructive, and an example of the construction is given in Figure \ref{fig:Munkres}. An alternative construction can be derived from the discussion in \cite[Chapter 2]{Hatcher2002} and is also illustrated in Figure \ref{fig:Hatcher}. We detail this construction in Appendix \ref{app:ChainHomotopy}. Moreover, one can ``stack'' these constructions on top of each other to obtain other admissible product complexes as also illustrated in Figure \ref{fig:stack}.
\end{remark}

\begin{definition}[Product complex inclusions]
Given a simplicial complex $X$ and admissible product complex $K$, we let $j_0, j_1: X \to K$ be the canonical inclusions of $X$ into $K$ at $t=0$ and $t=1$, respectively. Explicitly, $j_0$ and $j_1$ are the simplicial maps induced by the respective vertex maps $\widetilde{j}_0, \widetilde{j}_1 : X^0\to K^0$ such that $\widetilde{j}_0(v) = [(v,0)]$ and $\widetilde{j}_1(v) = [(v,1)]$ for any $v\in X^0$. 
\end{definition}

\begin{lemma}\label{lem:chain_homotopy_existence}
Let $X$ be a simplicial complex. Then, there exists an admissible product complex $K$ and map $E:\C_k(X) \to \C_{k+1}(K)$ such that
\begin{equation}\label{eq:chainhomotopy}
    \partial E + E \partial = (j_1)_\# - (j_0)_\# .
\end{equation}
\end{lemma}

The map $E$ provided by the above lemma is referred to as a chain homotopy between $(j_1)_\#$ and $(j_0)_\#$ in simplicial topology (see e.g. \cite[Section 12]{Munkres1984}) and is exactly the extrusion map we will use to construct the simplicial cone operator. We give a constructive proof of this lemma in Section \ref{sec:discPoincHom} using the product complexes illustrated in Figures \ref{fig:Hatcher} and \ref{fig:stack}. We note that alternative constructions are possible  (e.g. using the product complex of Figure \ref{fig:Munkres}), but we do not explore this here. 
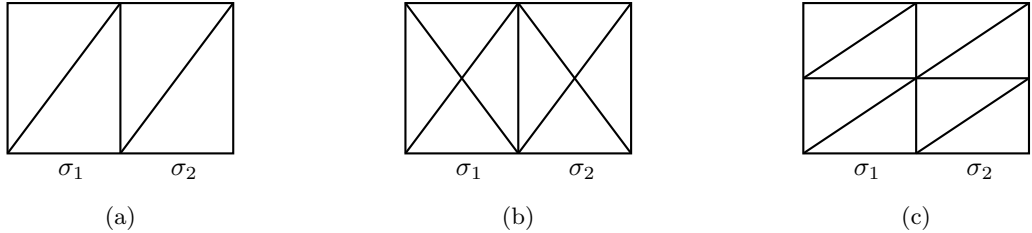
\begin{figure}[htbp]
    \centering
    \begin{subfigure}[b]{0.3\textwidth}
        \centering
        \begin{tikzpicture}[scale=1, thick]
            \coordinate (A) at (0,0);
            \coordinate (B) at (1.5,0);
            \coordinate (C) at (3,0);
            \coordinate (D) at (3,2);
            \coordinate (E) at (1.5,2);
            \coordinate (F) at (0,2);

            \draw (A) -- (C) -- (D) -- (F) -- cycle;
            \draw (B) -- (E);

            \draw (A) -- (E);
            \draw (B) -- (D);

            \node[below] at (0.85, 0) {$\sigma_1$};
            \node[below] at (2.35, 0) {$\sigma_2$};
        \end{tikzpicture}
        \caption{}\label{fig:Hatcher}
    \end{subfigure}
    \hfill
    \begin{subfigure}[b]{0.3\textwidth}
        \centering
        \begin{tikzpicture}[scale=1, thick]
            \coordinate (A) at (0,0);
            \coordinate (B) at (1.5,0);
            \coordinate (C) at (3,0);
            \coordinate (D) at (3,2);
            \coordinate (E) at (1.5,2);
            \coordinate (F) at (0,2);

            \draw (A) -- (C) -- (D) -- (F) -- cycle;
            \draw (B) -- (E);

            \draw (A) -- (E);
            \draw (B) -- (F);
            \draw (B) -- (D);
            \draw (C) -- (E);

            \node[below] at (0.85, 0) {$\sigma_1$};
            \node[below] at (2.35, 0) {$\sigma_2$};
        \end{tikzpicture}
        \caption{}\label{fig:Munkres}
    \end{subfigure}
    \hfill
    \begin{subfigure}[b]{0.3\textwidth}
        \centering
        \begin{tikzpicture}[scale=1, thick]
            \coordinate (A) at (0,0);
            \coordinate (B) at (1.5,0);
            \coordinate (C) at (3,0);
            \coordinate (D) at (3,2);
            \coordinate (E) at (1.5,2);
            \coordinate (F) at (0,2);
            
            \coordinate (M1) at (0,1);
            \coordinate (M2) at (1.5,1);
            \coordinate (M3) at (3,1);

            \draw (A) -- (C) -- (D) -- (F) -- cycle;
            \draw (B) -- (E);
            
            \draw (M1) -- (M3);

            \draw (A) -- (M2);
            \draw (B) -- (M3);
            \draw (M1) -- (E);
            \draw (M2) -- (D);

            \node[below] at (0.85, 0) {$\sigma_1$};
            \node[below] at (2.35, 0) {$\sigma_2$};
        \end{tikzpicture}
        \caption{}\label{fig:stack}
    \end{subfigure}
    \caption{Different admissible product complexes of $\Omega \times I$ where the complex $X$ consists of two edges $\sigma_1$ and $\sigma_2$ which meet at a vertex.}\label{fig:productspacecomplex}
\end{figure}

For the remainder of this section, we let $K$ and $E$ be an admissible product complex and extrusion map respectively, which exist due to Lemma \ref{lem:chain_homotopy_existence}.

\subsection{Simplicial cone operator on discrete contractible simplicial complexes}\label{sec:simp_cones}

We now provide a simplicial cone operator in the setting of the below definition.

\begin{definition}[Discrete contractibility]\label{def:discrete_contractibility}
Let $\id_{X}:X\to X$ be the simplicial map corresponding to the identity on $|X|$ and $\pi_{\C}:X\to [a]$ be the simplicial map corresponding to the vertex map $v\mapsto a$ for a given vertex $a\in X^0$. A simplicial map $\Psi: K \to X$ such that 
$$\Psi(\cdot, 0) = \pi_{\C}, \quad \Psi(\cdot,1) = \id_{X}$$
is called a \emph{discrete contraction}. If such a map exists, we say $X$ is a \emph{discrete contractible} simplicial complex.
\end{definition}
    
\begin{definition}[Simplicial cone operator on discrete contractible simplicial complexes]\label{def:gen_simp_cone}
Let $X$ be a discrete contractible simplicial complex, and let 
$\Psi: K \to X$ 
be a discrete contraction.
The simplicial cone operator (corresponding to $\Psi$) denoted $\Co : \C_k(X) \longrightarrow \C_{k+1}(X)$ is defined as follows. For any simplex $\sigma \in X^k$,
\begin{equation}
    \Co(\sigma) = (\Psi_\# \circ E)(\sigma).
\end{equation}
\end{definition}

\begin{theorem}\label{thm:simplicial-cone}
Let $X$ be a discrete contractible simplicial complex with discrete contraction $\Psi : K \to X$. Then, the operator, $\Co$, of Definition \ref{def:gen_simp_cone} corresponding to $\Psi$ is a simplicial cone operator.
\end{theorem}

\begin{proof}

We have for $\sigma\in X^k$, $k\geq 1$, 
\begin{alignat*}{2}
\partial\Co(\sigma) + \Co(\partial\sigma) &= \partial(\Psi_\# E)(\sigma) + (\Psi_\# E)(\partial \sigma)\\
&= \Psi_\#(\partial E + E\partial)(\sigma) \quad &&\text{by \eqref{eq:simplicial_boundary_sharp_commute}} \\
&= \Psi_\#\bigl((j_1)_\# - (j_0)_\#\bigr)(\sigma) \quad &&\text{by \eqref{eq:chainhomotopy}} \\
&=  (\Psi\circ j_1)_\#(\sigma) - (\Psi\circ j_0)_\#(\sigma) \quad &&\text{by \eqref{eq:simplicial_map_functoriality}} \\
&= (\id_{X})_\#(\sigma) - (\pi_\C)_\#(\sigma) = \sigma \quad &&\text{since $\Psi \circ j_1 = \id_{X}$ and $\Psi \circ j_0 = \pi_{\C}$},
\end{alignat*}
where we used that $(\pi_{\C})_{\#} \sigma = 0$ for all $\sigma\in X^k$ when $k\geq 1$ which follows from the definitions of $\pi_{\C}$ and the induced simplicial chain map. 

The $k=0$ case is similar, but now uses the fact that $(\pi_{\C})_\#(v) = \pi_{\C}(v) =  a$ for all $v\in X^0$, which again follows from the definitions of $\pi_{\C}$ and the induced simplicial chain map.
\end{proof}

We note that finding an explicit discrete contraction may be difficult in practice for an arbitrary simplicial complex. We give a way of constructing such discrete contractions in Section \ref{sec:strong_collapse_disc_contr} under certain assumptions on the simplicial complex. In comparison, it is often easier to define the smooth contraction for a given contractible domain $\Omega$. Hence, we provide in the next section a singular cone operator which leverages the smooth contraction through the use of singular simplices.

\subsection{Singular cone operator on Lipschitz contractible domains}\label{sec:Whitney_based_discrete_op}

\begin{definition}[Lipschitz contractible domain]
    We say a domain $\Omega\subset\R^n$ is \emph{Lipschitz contractible} to a point $a\in\Omega$ if there exists a Lipschitz map $\Phi:\Omega\times I \to \Omega$ where $I=[0,1]$ such that $\Phi(x,1) = x$ and $\Phi(x,0) = a$. 
\end{definition}

We now construct a singular cone operator in the setting of the above definition. To this end, recall that we can identify a geometric $k$-simplex, $\sigma$, with its linear singular representation $L\sigma$. With a slight abuse of notation, we let $L$ be an operator which acts on simplicial chains and outputs the corresponding linear singular chain. That is, $L:\C_k(X) \to \Ss_k^{\lin}(X)$ is defined on oriented simplices $\sigma\in X^k$ by $L(\sigma) = L\sigma$ and extended linearly. We now introduce the singular cone operator.

\begin{definition}[Singular cone operator on Lipschitz contractible domains]\label{def:Pgen}
    Let $\Omega$ be a Lipschitz contractible domain with Lipschitz contraction $\Phi:\Omega\times I \to \Omega$ which is the geometric realization of the simplicial complex $X$. We define $\Co^\Ss: \C_k(X) \to \Ss_{k+1}^{\sm}$ by setting
    \begin{equation}
        \Co^\Ss(\sigma) = (\Phi^\# \circ L\circ E) (\sigma).
    \end{equation}
\end{definition}

\begin{theorem}\label{lem:singHomotopy}
The operator defined in Definition \ref{def:Pgen} is a singular cone operator.
\end{theorem}

\begin{proof}
    For $\sigma\in X^k$, we compute using \eqref{eq:singular_boundary_sharp_commute} and \eqref{def:partialSslin},
    \begin{align*}
        \partial^\Ss \Co^\Ss(\sigma) &= \partial^\Ss \Phi^\# (L (E(\sigma))) = \Phi^\#(\partial^\Ss L(E(\sigma))) = \Phi^\#(L(\partial E(\sigma))).
    \end{align*}
    Therefore, by \eqref{eq:chainhomotopy}, we have
    \begin{align*}
        \Co^\Ss(\partial \sigma)+ \partial^\Ss \Co^\Ss(\sigma) &= \Phi^\#(L(E(\partial\sigma) + \partial E(\sigma)))  \\
        &= \Phi^\#(L((j_1)_\#(\sigma) - (j_0)_\#(\sigma))) = \Phi \circ L[(\sigma, 1)] - \Phi\circ  L[(\sigma, 0)],
    \end{align*}
    where in the last step we used the definitions of $j_1$ and $j_0$ and the induced singular chain map. Then, from the definition of $\Phi$, we have that $\Phi \circ L[(\sigma, 1)] = L\sigma$ and $\Phi \circ L[(\sigma, 0)] = 0$ for $k\geq 1$, while $\Phi \circ L[(v, 0)] = L[a]$ for any $v\in X^0$.
\end{proof}

In the case that $\Omega$ is star-shaped with respect to the point $a$, we can in fact recover the singular cone operator of Definition \ref{def:sing_cone_star}. We detail this here. Let $\Omega$ be star shaped with respect to a point $a$. In this case, we have the explicit formula for the contraction: $\Phi(x,t) = (1-t)a + tx$.

We use the explicit product complex and extrusion given in Example \ref{ex:example_product_complex} of Section \ref{sec:discPoincHom} below. An example of this product complex is also depicted in Figure \ref{fig:Hatcher} above. For a simplex $\sigma = [v_0, \dots, v_k]$, the extrusion is given by:
$E(\sigma) = \sum_{i=0}^k (-1)^i \tau_i$
where $\tau_i = [(v_0,0), \dots, (v_i,0), (v_i,1), \dots, (v_k,1)]$.
Then, we can see that
\begin{alignat*}{1}
    \Co^\Ss(\sigma) &= \Phi \circ L(E(\sigma)) = \sum_{i=0}^k (-1)^i \Phi \circ L\tau_i\\
    &= \sum_{i=0}^k (-1)^i L[\Phi(v_0, 0), \dots, \Phi(v_i, 0), \Phi(v_i, 1), \dots, \Phi(v_k, 1)]\\
    &= L[a, v_0, \dots, v_k] - L[a, a, v_1, \dots, v_k] + \cdots + (-1)^{k} L[a, \dots,a, v_k].
\end{alignat*}
where we used crucially that $\Phi$ is a linear map in the second to last step. 

When one integrates over this singular chain, the only possibly nondegenerate term is the first term $L[a, v_0, \dots, v_k]$ which is exactly the operator in Definition \ref{def:sing_cone_star}.

\subsection{Connections between the discrete Poincar\'e operators}\label{sec:comparison}

We make some remarks comparing the two discrete Poincar\'e operators obtained from the simplicial cone operator of Definition \ref{def:gen_simp_cone} and the singular cone operator of Definition \ref{def:Pgen}.

Consider a cochain $\alpha \in \C^k$ and simplex $\sigma\in X^{k-1}$. Let $P$ be the discrete Poincar\'e operator obtained from the simplicial cone operator of Definition \ref{def:gen_simp_cone} and let $\mathsf{P}$ be the discrete Poincar\'e operator on cochains obtained from the singular cone operator of Definition \ref{def:Pgen}. Then, the two operators are such that
\begin{alignat}{1}
    \bl \mathsf{P} \alpha, \sigma \br &= \bl \dr \mathsf{P} W \alpha, \sigma \br = \int_{\Co^\Ss(\sigma)} W \alpha \label{defPPhi}\\
    \bl P \alpha, \sigma \br &= \bl \alpha, \Co(\sigma) \br = \int_{\Co(\sigma)} W\alpha. \label{defPDEC}
\end{alignat}
where we used \eqref{aux3} in the last equality. Note that the integral in \eqref{defPPhi} is over a singular chain while the integral in \eqref{defPDEC} is over a simplicial chain. Hence, the connection between the two operators reduces to the connection between $\Co^\Ss = \Phi^\# \circ L \circ E$ and $\Co = \Psi_\# \circ E$. 

Now, recall that $\Psi: K\to X$ is a simplicial map and so is in particular a piecewise linear map $\Psi: \Omega\times I \to \Omega$. This map $\Psi$ is thus a valid Lipschitz continuous contraction. Then, one can see that if $\Psi = \Phi$, the two operators will give the same result.

As a simple example, consider the simplicial complex $X$ which is the closed star of a vertex $a$. That is, $a$ is a vertex in the simplicial complex such that every other simplex in $X$ is a face of a simplex which contains $a$. The geometric realization of such a complex will always be star-shaped with respect to $a$, and hence, we can consider the straight-line contraction $\Phi(x,t)=(1-t)a+tx$. Because this straight-line contraction maps the vertices of the product complex of Definition \ref{def:pc} strictly to valid simplices in $X$, it acts as a valid simplicial discrete contraction $\Psi: K \to X$. In this setting, the resulting discrete Poincar\'e operators using $\Co^\Ss$ and $\Co$ will be equivalent.

\section{Explicit Product Complexes, Extrusions, and Discrete Contractions}\label{sec:discPoincHom}

In this section, we give a constructive proof of Lemma~\ref{lem:chain_homotopy_existence} and discuss the discrete contractibility assumption. We first provide explicit product complexes in~\S\ref{sec:product_complexes}. Then, using these product complexes, we define $E$ operators on chains which satisfy \eqref{eq:chainhomotopy} in~\S\ref{subsec:extrusions}. Finally, we show how to construct discrete contractions under strong collapsibility assumptions in Section \ref{sec:strong_collapse_disc_contr}. 

\subsection{Admissible Product Complex}\label{sec:product_complexes}

We begin by providing some notions useful to the construction of the admissible product complexes.

\begin{definition}[Locally Total Order]
A \emph{locally total order} on vertex set $X^0$ of a simplicial complex $X$ is a partial order $\leq$ on $X^0$ such that for each simplex $\sigma \in X$, $\leq$ is a total order on $\sigma^0$. A complex $X$ is called \emph{ordered} if it is equipped with a locally total order on its vertex set $X^0$. A simplex $\sigma = [v_0,\dots,v_k]$ in a ordered complex $X$ is called \emph{ordered} if $v_0 < v_1 < \cdots < v_k$. 
\end{definition}

Let $X$ be a simplicial complex and let $I = [0,1]$. Given the vertex set $X^0$ and a locally total order, then the vertex set  $X^0 \times I^0$, regarded as a poset with the Cartesian product partial order, can be equipped with a simplicial complex structure where simplices are totally ordered subsets.

\begin{definition}[A particular admissible product complex]\label{def:pc}
Let $X$ be an ordered complex, and let $I = [0,1]$ denote the $1$–simplex with ordered vertices $0 < 1$.
A particular admissible product complex is the simplicial complex whose vertices are pairs in $X^0 \times I^0$, and whose simplices are the finite
subsets $A \subseteq X^0 \times I^0$ that are totally ordered with respect to the product order
\[
(x,i) \leq (x',i') \quad \text{iff} \quad x \leq_X x' \text{ and } i \leq_I i'.
\]
\end{definition} 
In fact, this is a special case of the so-called staircase triangulation of the product of two simplices (see e.g. \cite[Chapter 6.2.3]{DeRaSa2010}).

\begin{definition}\label{def:nstep}
For a positive integer $n$, let $I_n$ denote the $(n-1)$-fold subdivision of the $1$-simplex $[0,1]$, whose maximal simplices are $\bigl[\tfrac{i}{n}, \tfrac{i+1}{n}\bigr]$, for $i = 0, 1, \dots, n-1$. The \emph{$n$-step} product complex of $\Omega \times I$ is denoted $K_{n}$ and is defined as
\[
K_n := \bigcup_{i=0}^{n-1} X \times \bigl[\tfrac{i}{n}, \tfrac{i+1}{n}\bigr],
\]
i.e., the simplicial complex obtained by stacking $n$ copies of the product complex from Definition \ref{def:pc} along the intervals of $I_n$. A simple example of $K_2$ for a simplicial complex $X$ consisting of two edges is depicted in Figure \ref{fig:stack}.
\end{definition}

\subsection{Extrusion}\label{subsec:extrusions}
Inspired by the idea in Remark~\ref{rem:DEC-hook} we define a chain homomorphism version of the extrusion of Remark~\ref{rem:DEC-hook} and then check that it satisfies the required~\eqref{eq:chainhomotopy}. This is also called the \emph{prism operator}~\cite[page 112]{Hatcher2002}.

\begin{definition}[Extrusion]\label{def:extrusion}
Given a simplicial complex $X$, let $K$ be the $n$-step product complex $K_n$ of Definition \ref{def:nstep}. Then, the extrusion map $E \colon \C_k(X) \to \C_{k+1}(K)$ is defined on oriented simplices of $X$ and extended linearly. For an oriented simplex $\sigma = [v_0,\dots,v_k]$, its extrusion is the chain
\[
E(\sigma) = \sum_{\tau \in K_\sigma^{(k+1)}} \sgn (\tau;\,K_\sigma)\,\tau,
\]
where $(K_\sigma)^{k+1}$ denotes the set of oriented $(k+1)$-simplices in the subcomplex $K_\sigma = \sigma\times I_n$ from Definition \ref{def:apc}. 

The sign factor $\sgn(\tau;\,K_\sigma) \in \{\pm 1\}$ aligns the orientation of $\tau$ with the product space. Specifically, we equip the subcomplex $K_\sigma$ with the product orientation (ensuring the inclusion $\sigma \hookrightarrow \sigma \times \{1\}$ is orientation-preserving). Meanwhile, each simplex $\tau$ receives a lexicographic orientation derived from the ordered vertices of $\sigma$ and $I_n$. The sign is $+1$ if these two orientations match, and $-1$ otherwise.
\end{definition}

\begin{example}\label{ex:example_product_complex}
Let $X$ be a simplicial complex. We consider the product complex $K_1$ ($1$-step product complex). Fix an ordered simplex $\sigma = [v_0,v_1,\dots,v_k]\in X^k$. We denote the $(k+1)$-simplices in $K_1$ as $\tau_i := [(v_0,0),\dots,(v_i,0)(v_i,1),\dots,(v_k,1)]$, for any $i\in\{0,\dots,k\}.$
Then $E(\sigma) = \sum_{i=0}^k (-1)^i \tau_i$. When $k=1$ and $\sigma = [v_0,v_1]$, the subcomplex $\sigma \times I$ has maximal simplices $[(v_0,0)(v_1,0)(v_1,1)]$ and $[(v_0,0)(v_0,1)(v_1,1)]$, and we have $E(\sigma) = [(v_0,0)(v_0,1)(v_1,1)] - [(v_0,0)(v_1,0)(v_1,1)].$
\end{example}

\begin{proposition}\label{prop:ChainHomotopy}
The operator defined in Definition \ref{def:extrusion} is a valid map for Lemma \ref{lem:chain_homotopy_existence} (i.e. it satisfies~\eqref{eq:chainhomotopy}).
\end{proposition}

The proof of the Proposition \ref{prop:ChainHomotopy} can be found in Appendix \ref{app:ChainHomotopy}.

\subsection{Strong Collapsibility and Discrete Contractions}\label{sec:strong_collapse_disc_contr}

We now show an example of how to construct a discrete contraction on so-called \textit{strongly collapsible} (or LC-reducible) simplicial complexes \cite{Matousek08, Barmak2012}. Note that the construction in Section \ref{sec:collapse_cone} based on simplicial collapse does not immediately give rise to a global simplicial map that acts as a discrete contraction. Hence, to explicitly construct the discrete contraction $\Psi: K \to X$ required by the simplicial cone operator of Section \ref{sec:simp_cones}, we instead restrict our attention to the strictly smaller class of strongly collapsible simplicial complexes. Essentially, the notion of strong collapsibility provides the structure to induce valid simplicial maps during a collapse operation. 

\begin{definition}[Dominated Vertex and Strong Collapse]
Let $X$ be a simplicial complex. A \emph{maximal simplex} of $X$ is a simplex of $X$ which is not a face of any simplex other than itself. A vertex $v \in X^0$ is said to be \emph{dominated} by another vertex $v' \in X^0$ if every maximal simplex of $X$ that contains $v$ also contains $v'$. 

An \emph{elementary strong collapse} is the operation of obtaining a subcomplex $X' \subset X$ by removing a dominated vertex $v$ and all open simplices containing it. We denote this by $X \strongcollapse X'$. A complex $X$ is \emph{strongly collapsible} if there exists a sequence of elementary strong collapses $X=X_0\strongcollapse X_1 \strongcollapse \cdots \strongcollapse X_m = a$ reducing $X$ to a single vertex $a \in X^0$. 
\end{definition}

\begin{figure}[htpb]
    \centering
    \begin{tikzpicture}[scale=0.95]
        \tikzset{
            vertex/.style={circle, fill=black, inner sep=1.2pt},
            edge/.style={thick, black},
            face/.style={fill=blue!10}
        }

        \newcommand{\strongcollapselabel}{$\searrow\mkern-8mu\searrow$}

        \begin{scope}[xshift=0cm]
            \coordinate (A) at (0,0);
            \coordinate (V1) at (1, 0);
            \coordinate (V2) at (0.5, 1);
            \coordinate (V3) at (-0.5, 1);
            \coordinate (V4) at (-1, 0);

            \fill[face] (A) -- (V1) -- (V2) -- cycle;
            \fill[face] (A) -- (V2) -- (V3) -- cycle;
            \fill[face] (A) -- (V3) -- (V4) -- cycle;

            \draw[edge] (V1) -- (V2) -- (V3) -- (V4);
            \draw[edge] (A) -- (V1);
            \draw[edge] (A) -- (V2);
            \draw[edge] (A) -- (V3);
            \draw[edge] (A) -- (V4);

            \node[vertex, label=below:$a$] at (A) {};
            \node[vertex, label=below:$v_1$] at (V1) {};
            \node[vertex, label=above:$v_2$] at (V2) {};
            \node[vertex, label=above:$v_3$] at (V3) {};
            \node[vertex, label=below:$v_4$] at (V4) {};
            \node[below] at (0,-0.6) {$X_4=X$};
        \end{scope}

        \node at (1.65, 0.4) {\strongcollapselabel};

        \begin{scope}[xshift=3.3cm]
            \coordinate (A) at (0,0);
            \coordinate (V1) at (1, 0);
            \coordinate (V3) at (-0.5, 1);
            \coordinate (V4) at (-1, 0);

            \fill[face] (A) -- (V3) -- (V4) -- cycle;

            \draw[edge] (V3) -- (V4);
            \draw[edge] (A) -- (V1); 
            \draw[edge] (A) -- (V3);
            \draw[edge] (A) -- (V4);

            \node[vertex, label=below:$a$] at (A) {};
            \node[vertex, label=below:$v_1$] at (V1) {};
            \node[vertex, label=above:$v_3$] at (V3) {};
            \node[vertex, label=below:$v_4$] at (V4) {};
            \node[below] at (0,-0.6) {$X_3$};
        \end{scope}

        \node at (4.95, 0.4) {\strongcollapselabel};

        \begin{scope}[xshift=6.6cm]
            \coordinate (A) at (0,0);
            \coordinate (V1) at (1, 0);
            \coordinate (V3) at (-0.5, 1);

            \draw[edge] (A) -- (V1); 
            \draw[edge] (A) -- (V3); 

            \node[vertex, label=below:$a$] at (A) {};
            \node[vertex, label=below:$v_1$] at (V1) {};
            \node[vertex, label=above:$v_3$] at (V3) {};
            \node[below] at (0,-0.6) {$X_2$};
        \end{scope}

        \node at (8.15, 0.4) {\strongcollapselabel};

        \begin{scope}[xshift=9.7cm]
            \coordinate (A) at (0,0);
            \coordinate (V3) at (-0.5, 1);

            \draw[edge] (A) -- (V3);

            \node[vertex, label=below:$a$] at (A) {};
            \node[vertex, label=above:$v_3$] at (V3) {};
            \node[below] at (0,-0.6) {$X_1$};
        \end{scope}

        \node at (11.0, 0.4) {\strongcollapselabel};

        \begin{scope}[xshift=12.2cm]
            \coordinate (A) at (0,0);

            \node[vertex, label=below:$a$] at (A) {};
            \node[below] at (0,-0.6) {$X_0 = a$};
        \end{scope}

    \end{tikzpicture}
    \caption{A strong collapse sequence $X=X_4\strongcollapse X_3 \strongcollapse \cdots \strongcollapse X_0 = a$. Because $a$ belongs to every maximal simplex, it dominates all other vertices. In the first step $X_4 \strongcollapse X_3$, removing the dominated vertex $v_2$ simultaneously removes two $2$-simplices. Moreover at each strong collapse step, a vertex is removed. Both of these facts are in contrast to an elementary simplicial collapse.}
    \label{fig:strong_collapse_sequence}
\end{figure}
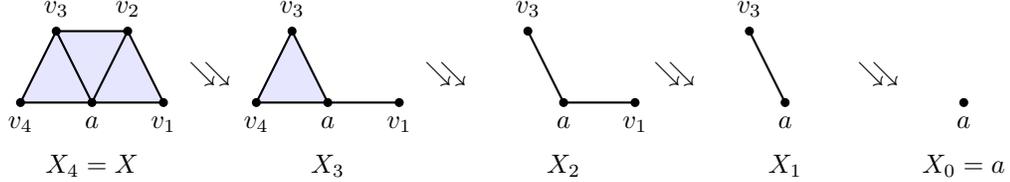

The critical advantage of a strong collapse over a standard simplicial collapse is that each elementary strong collapse naturally induces a simplicial map. If $v$ is dominated by $v'$, the map $r: X^0 \to (X')^0$ defined by $r(v) = v'$ and $r(u) = u$ for all $u \neq v$ is a valid simplicial map. We use this property to construct the discrete contraction.

Let $X$ be a strongly collapsible simplicial complex. We will show that $X$ is discretely contractible by constructing a discrete contraction $\Psi: K_m \to X$ where $K_m$ is the $m$-step admissible product complex defined in Definition \ref{def:nstep}.

Let $X = X_m \strongcollapse X_{m-1} \strongcollapse \dots \strongcollapse X_0 = a$ be a strong collapse sequence. For each step $1 \le k \le m$, the complex $X_{k-1}$ is obtained from $X_k$ by removing a vertex $v_k$ that is dominated by some vertex $v_k'$ in $X_k$. 

Let $r_k: X_k \to X_{k-1}$ be the simplicial map such that $r_k(v_k) = v_k'$ and $r_k(u) = u$ for all vertices $u \neq v_k$. We define a sequence of simplicial maps $f_k: X \to X_k$ recursively by letting  
\begin{align*}
    f_m := \id_X, \text{ and } f_{k-1} := r_{k} \circ f_{k} \quad \forall k\in \{1, \dots, m\}.
\end{align*}
Because the composition of simplicial maps is simplicial, each $f_k: X \to X_k$ is a valid simplicial map. Note that $f_0(u) = a$ for all $u \in X^0$, so $f_0 = \pi_{\mathcal{C}}$.

\begin{definition}[Discrete contraction from strong collapse]\label{def:discrete_contraction}
Let $K_m$ be the $m$-step product complex of Definition \ref{def:nstep}. We define the simplicial map $\Psi: K_m \to X$ by setting its action on vertices as:
\begin{equation*}
    \Psi(v, t_k) = f_{m-k}(v) \quad \forall (v, t_k) \in X^0\times I_m^0.
\end{equation*}
\end{definition}

\begin{proposition}
    The operator $\Psi$ defined in Definition \ref{def:discrete_contraction} is a discrete contraction. That is, $\Psi$ is a simplicial map such that 
$$\Psi(\cdot, 0) = \pi_{\C}, \quad \Psi(\cdot,1) = \id_{X}.$$
\end{proposition}

\begin{proof}
We must verify that $\Psi$ induces a valid simplicial map on the entirety of $K_m$. Let $\Sigma \in K_m^p$ be an arbitrary $p$-simplex in the product complex $K_m$. By Definition \ref{def:nstep}, $\Sigma$ must lie entirely within a single slab $X \times [t_{k-1}, t_k]$ for some time step $k\in\{1, \dots, m\}$.

By the total order defined on the product complex (Definition \ref{def:pc}), the vertices of $\Sigma$ must form a strictly increasing sequence $v_0 < v_1 < \dots < v_p$, where each $v_\ell = (y_\ell, \tau_\ell) \in X^0 \times I_m^0$. Because the temporal coordinates $\tau_\ell \in \{t_{k-1}, t_k\}$ can increase at most once, the spatial coordinates form a non-decreasing sequence $y_0 \le_X y_1 \le_X \dots \le_X y_p$ that can contain at most one repeated vertex (which occurs if and only if the temporal coordinate increases at that exact step). Consequently, the set of distinct spatial vertices $\{y_0, \dots, y_p\}$ spans a valid simplex $\sigma \in X$ of dimension either $p$ or $p-1$.  
Applying the vertex map $\Psi$ to $\Sigma^0$ yields the set:
\begin{equation*}
    \Psi(\Sigma^0) = \{\Psi(y_0, \tau_0), \dots, \Psi(y_p, \tau_p)\}.
\end{equation*}
Because $\tau_\ell \in \{t_{k-1}, t_k\}$, each $\Psi(y_\ell, \tau_\ell)$ evaluates to either $f_{k-1}(y_\ell)$ or $f_k(y_\ell)$. Thus, the mapped set is strictly a subset of $f_{k-1}(\sigma^0) \cup f_k(\sigma^0)$.

Now, to prove that $\Psi$ is a simplicial map, we must show that $\Psi(\Sigma^0)$ spans a valid simplex in $X$. Let $s := f_k(\sigma)$, which is a simplex in $X_k$. Then $r_k(s) = f_{k-1}(\sigma)$. Therefore, $\Psi(\Sigma^0)$ is contained in the vertex set of $s \cup r_k(s)$.
Hence, we now show that $s \cup r_k(s)$ spans a valid simplex in $X_k$ by examining the action of the map $r_k$ on $s$. To this end, we recall that $v_k$ is the dominated vertex in $X_k$ and consider two cases.
\begin{enumerate}
    \item  If $v_k \notin s^0$, then $r_k$ acts as the identity on $s$. Thus, $s \cup r_k(s) = s \cup s = s$, which is trivially a valid simplex in $X_k$.
    \item If $v_k \in s^0$, we can write $s = [\eta, v_k]$ where $\eta$ is the face of $s$ not containing $v_k$. Because $v_k$ is dominated by $v_k'$ in $X_k$, there exists a maximal simplex $\Sigma' \in X_k$ containing $s$. Moreover, by the definition of domination, $v_k'$ is also a vertex of $\Sigma'$. Therefore,  $[s, v_k'] = [\eta, v_k, v_k']$ is a face of $\Sigma'$ and is consequently a valid simplex in $X_k$. We conclude by noting that $r_k(s) = [\eta, v_k']$ so $s \cup r_k(s)$ is precisely $[\eta, v_k, v_k']$.
\end{enumerate}

In all cases, the mapped vertices $\Psi(\Sigma^0)$ form a subset of a valid simplex in $X_{k} \subseteq X$. Thus, $\Psi: K_m \to X$ is a well-defined simplicial map.
Furthermore, $\Psi(\cdot, 1) = f_m = \id_X$ and $\Psi(\cdot, 0) = f_0 = \pi_{\mathcal{C}}$ as required. 
\end{proof}

\section{The discrete Bogovski\u{\i} operator on Whitney forms}\label{sec:Bog}

We now suppose the domain $\Omega$ is star-shaped with respect to a point $a\in\Omega$ to define the discrete Bogovski\u{\i} operator. With Section \ref{sec:discPoincWhitsimple} in mind, for $\sigma = [v_0, \dots, v_k]$, we use as a definition:
\begin{equation*}
    \Co^\Ss(\sigma) = L[a, v_0, \dots, v_k].
\end{equation*}

\subsection{Singular Infinite Cones}

In order to define the Bogovski\u{\i} operator, we will need to introduce a notion of singular infinite cones. To this end, we let $O_k$ be the nonnegative $k$-orthant embedded in $\R^n$. Explicitly, this is the set
\begin{equation*}
    O_k = \left\{(x_1, \dots, x_n)\in\R^n: x_i \geq 0 \text{ for } 1\leq i\leq k \text{ and } x_i = 0 \text{ for } k+1\leq i \leq n \right\}.
\end{equation*}
For a given topological space, $E$, we let a singular cone be any mapping $T:O_k\to E$. We will take $E=\R^n$ in all that follows. All we will need are the linear singular cones. To formalize this, we first define for an arbitrary set of points $a_0, \dots, a_k$ in $\R^n$,  a corresponding infinite cone
\begin{equation*}
    [a_0, \dots, a_k]^{\vee} = \left\{a_0+\sum_{i=1}^k c_i (a_i-a_0): c_i\in\R_{\ge0}, \forall\, 1\leq i \leq k\right\} ,
\end{equation*}
where $\mathbb{R}_{\ge0}$ is the set of nonnegative real numbers.

Now, for such $a_0, \dots, a_k$ in $\R^n$, there is a unique affine map $L^\vee[a_0, \dots, a_k]:O_k \to [a_0, \dots, a_k]^{\vee}$ which sends the origin to $a_0$ and maps the rays of the orthant which start at the origin and go in the direction of the principal axes to the rays which start at $a_0$ and go in the direction of the vectors $(a_i-a_0)$. We call $L^\vee[a_0, \dots, a_k]$ the linear singular cone determined by $\{a_0, \dots, a_k\}$. In particular, $L^\vee[a_0, \dots, a_k]$ has the same explicit representation as \eqref{eq:lrep} but with the domain now being $O_k$. Some illustrations are given in Figure \ref{fig:singularcones}.



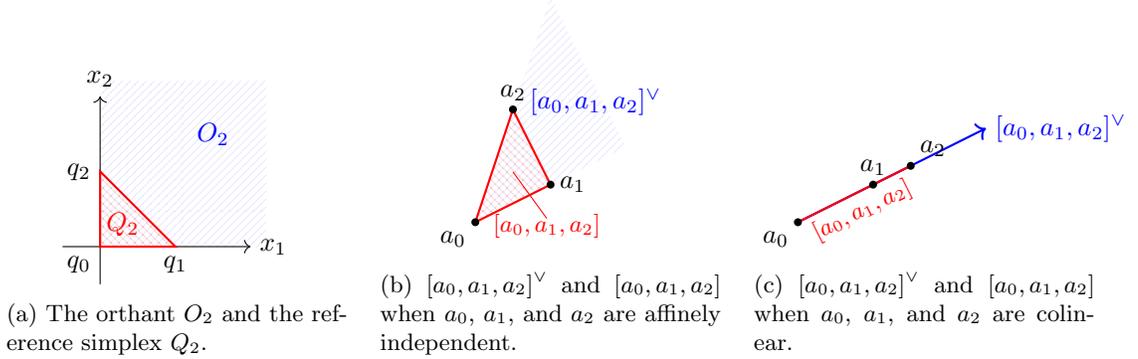
\begin{figure}[ht]
    \centering
    \begin{subfigure}[b]{0.3\textwidth}
        \centering
        \begin{tikzpicture}[scale=1]
            \fill[pattern=north east lines, pattern color=blue, opacity=0.3] (0,0) rectangle (2.2,2.2);
            \node[blue] at (1.5,1.5) {$O_2$};

            \draw[->] (-0.5,0) -- (2,0) node[right] {$x_1$};
            \draw[->] (0,-0.5) -- (0,2) node[above] {$x_2$};

            \fill[pattern=north west lines, pattern color=red, opacity=0.4] (0,0) -- (1,0) -- (0,1) -- cycle;
            \draw[thick, red] (0,0) -- (1,0) -- (0,1) -- cycle;
            \node[red] at (0.3,0.3) {$Q_2$};
            \node[below left] at (0,0) {$q_0$};
            \node[below] at (1,0) {$q_1$};
            \node[left] at (0,1) {$q_2$};
        \end{tikzpicture}
        \caption{The orthant $O_2$ and the reference simplex $Q_2$.}
        \label{fig:orthant_simplex}
    \end{subfigure}
    \hspace{2mm} 
    \begin{subfigure}[b]{0.3\textwidth}
        \centering
        \begin{tikzpicture}[scale=1]
            \coordinate (a0) at (0,0);
            \coordinate (a1) at (1,0.5);
            \coordinate (a2) at (0.5,1.5);

            \fill[pattern=north east lines, pattern color=blue, opacity=0.3] (a0) -- (2.0,1.0) -- (1.0,3.0) -- cycle;
            \node[blue] at (1.6,1.6) {$[a_0, a_1, a_2]^\vee$};

            \fill[pattern=north west lines, pattern color=red, opacity=0.4] (a0) -- (a1) -- (a2) -- cycle;
            \draw[thick, red] (a0) -- (a1) -- (a2) -- cycle;
            
            \coordinate (center) at (0.5, 0.67); 
            \coordinate (labelpos) at (0.95, -0.05); 

            \node[red] at (labelpos) {\small $[a_0, a_1, a_2]$};
            \draw[thin, red] (0.95, 0.05) -- (center); 

            \fill (a0) circle (1.5pt) node[below left] {$a_0$};
            \fill (a1) circle (1.5pt) node[right] {$a_1$};
            \fill (a2) circle (1.5pt) node[above] {$a_2$};
        \end{tikzpicture}
        \caption{$[a_0,a_1,a_2]^\vee$ and $[a_0, a_1, a_2]$ when $a_0$, $a_1$, and $a_2$ are affinely independent.}
        \label{fig:non_degenerate_cone}
    \end{subfigure}
    \hspace{2mm} 
    \begin{subfigure}[b]{0.3\textwidth}
        \centering
        \begin{tikzpicture}[scale=1]
            \coordinate (a0) at (0,0);
            \coordinate (a1) at (1,0.5);
            \coordinate (a2) at (1.5,0.75);

            \draw[thick, blue, ->] (a0) -- (2.5,1.25);
            \node[blue, right] at (2.5,1.25) {$[a_0, a_1, a_2]^\vee$};

            \draw[thick, red] (a0) -- (a2) node[midway, below, sloped, red, inner sep=2pt] {\small $[a_0, a_1, a_2]$};

            \fill (a0) circle (1.5pt) node[below left] {$a_0$};
            \fill (a1) circle (1.5pt) node[above] {$a_1$};
            \fill (a2) circle (1.5pt) node[above right] {$a_2$};
        \end{tikzpicture}
        \caption{$[a_0,a_1,a_2]^\vee$ and $[a_0, a_1, a_2]$ when $a_0$, $a_1$, and $a_2$ are colinear.}
        \label{fig:degenerate_cone}
    \end{subfigure}
    \caption{Examples of different scenarios for the images of infinite singular cones.}
    \label{fig:singularcones}
\end{figure}


The orientation of any $L^\vee[a_0, \dots, a_k]$ is given in the same way as for simplices based on the ordering of $a_0, \dots, a_k$. We denote the space of oriented linear singular $k$-cones by $\Xi_k^\vee$. 

\begin{definition}[Singular Cone Chains]
    The space of linear singular $k$-cone chains is then the vector space
\begin{equation*}
    \Ss^\vee_k = \left\{\sum_{L\in \Xi_k^{\vee}} c_L L: c_L = 0 \text{ for all but finitely many }L\in\Xi_k^{\vee}\right\}.
\end{equation*}
\end{definition}

Now we define the boundary operator $\partial^{\Ss^{\vee}}: \Xi_k^{\vee} \to\Ss_{k-1}^{\vee}$ by
\begin{equation}\label{def:partialSsvee}
    \partial^{\Ss^{\vee}} L^{\vee}[a_0, \dots, a_k] =  \sum_{i=1}^k (-1)^i L^{\vee}[a_0, \dots, \widehat{a_i}, \dots, a_k],
\end{equation}
for any set of $k+1$ points $ a_0, \dots, a_k $ in $\R^n$. The definition is then extended to all of $\Ss_k^{\vee}$ by linearity.

Because the elements in $\Xi_k^\vee$ have the same representation as those in $\Xi_k$, we can define integration over singular cones in essentially the same way as for singular simplices (cf. \eqref{def:intSing}), but now accounting for the different domain $O_k$. 

Explicitly, the integral of $\omega\in L^1\Lambda^k(\R^n)$ over $L^\vee$ is
\begin{equation}\label{def:intSingC}
    \int_{L^\vee} \omega := \int_{O_k} (L^\vee)^*\omega.
\end{equation}

\begin{remark}[Degenerate Cones]\label{rem:degC}
    As in Remark \ref{rem:deg}, when the points $a_0, \dots a_k$ are not affinely independent, the Jacobian $DL^\vee[a_0, \dots a_k]$ is rank deficient, and so the integral in \eqref{def:intSingC} is zero.
\end{remark}

\begin{lemma}\label{lem:equal}
    Let $a_0, \dots, a_k \in \R^n$ and $\omega$ be a $k$-form over $\R^n$. Then 
    \begin{equation}
        (L^\vee[a_0, \dots, a_k])^* \omega|_{Q_k} = L[a_0, \dots, a_k]^*\omega.
    \end{equation}
\end{lemma}
\begin{proof}
    This is immediate since the mappings corresponding to $L[a_0, \dots, a_k]$ and $L^\vee[a_0, \dots, a_k]$ are the same.
\end{proof}

We will need the following version of Stokes' theorem on singular cone chains. 

\begin{lemma}
     Let $\omega\in L^1\Lambda^{k-1}(\R^n)$ with compact support. Then, for any $C\in \Ss_k^\vee$, it holds that
\begin{equation}\label{eq:StokesSingC}
    \int_C d \omega = \int_{\partial^{\Ss^\vee} C} \omega.
\end{equation}
\end{lemma}
\begin{proof}

It suffices to prove the result for $C = L^\vee[a_0, \dots, a_k]$ for arbitrary points $a_0, \dots, a_k$ in $\R^n$. This result follows by considering a ball $B$ which is large enough to contain the support of $\omega$, and applying the Stokes' theorem on singular simplices, \eqref{eq:StokesSing}, to the singular simplex corresponding to the points in $\{a_0\}\cup(\partial B \cap \image(C))$.
\end{proof}

\subsection{An Infinite Cone Operator}
For the discrete Bogovski\u{\i} operator, we will need the following operator $\Co^{\vee}:\C_k(X) \to \Ss_{k+1}^{\vee}$ which is such that
\begin{equation*}
    \Co^{\vee}(\sigma) = L^\vee[a,v_0, \dots, v_k]
\end{equation*}
for any $\sigma=[v_0,\dots, v_k]\in X^{(k)}$. The key result about $\Co^\vee$ is the following. 
\begin{lemma}\label{lem5.3}
It holds, 
\begin{subequations}
\begin{alignat}{2}
    \Co^{\vee}(\partial \sigma)+ \partial^{\Ss^{\vee}} \Co^{\vee}(\sigma) &= 0 \qquad && \forall \sigma \in X^{(k)},\, k\geq 1,\label{hom2a}\\
    \partial^{\Ss^{\vee}} \Co^{\vee}(\sigma) &= -L[a] \qquad && \forall \sigma \in X_0.\label{hom2b}
\end{alignat}
\end{subequations}
\end{lemma}

\begin{proof}
First, let $k\geq 1$ and $\sigma=[v_0, \ldots, v_k]$. Then
\begin{align*}
     \partial^{\Ss^{\vee}} \Co^{\vee}(\sigma) &= \partial^{\Ss^{\vee}} L^\vee[a,v_0, \dots, v_k]\\
     &= \sum_{i=0}^k (-1)^{i+1} L^\vee[a, v_0, \dots, \widehat{v_i}, \dots, v_k] = - \Co^{\vee}(\partial \sigma).
\end{align*}
On the other hand for $\sigma\in X^{(0)}$, $\partial^{\Ss^{\vee}} \Co^{\vee}(\sigma) = \partial^{\Ss^{\vee}} L^\vee[a, \sigma] = -L^\vee[a] = -L[a]$ where we used that $L^\vee[a_0]= L[a_0]$ for any single point $a_0\in\R^n$.
\end{proof}

\subsection{The Discrete Bogovski\u{\i} Operator}
First, note that if $u \in \mathring{V}_h^k$ then we can extend it by zero to all of $\R^n$ and it will belong to $H\Lambda^k(\R^n)$. Let $u \in \mathring{V}_h^k$ we define $\Bo u \in V_h^{k-1}$ by setting
\begin{alignat}{1}\label{def:Bo}
\int_{\sigma} \Bo u= \int_{\Co^\Ss(\sigma)} u -\int_{\Co^\vee(\sigma)} u\qquad \sigma \in X^{(k-1)}
\end{alignat}

We will need an additional assumption for the $k=n$ case and so we first note the result for the $k<n$ case below.

\begin{theorem}\label{thm:Bog}
For all $u \in \mathring{V}_h^k$, $0\leq k \leq n-1$, it holds that $\Bo  u \in \mathring{V}_h^{k-1}$. Moreover, 
\begin{equation}\label{eq:Bog}
    \Bo d +  d \Bo =\text{id}.
\end{equation}
\end{theorem}
\begin{proof}

We first show that $\Bo  u \in \mathring{V}_h^{k-1}$. To this end, let $\sigma\in X^{(k-1)}$ be such that $\sigma\subset\partial\Omega$. Then, we can compute using \eqref{def:Bo} and Lemma \ref{lem:equal}
\begin{align*}
    \int_{\sigma} \Bo u &= - \left(\int_{O_k} (L^\vee[a\sigma])^*u -\int_{Q_k} (L[a\sigma])^* u\right) = - \int_{O_k\setminus Q_k} (L^\vee[a\sigma])^*u.
\end{align*}
Since $u$ is supported in $\Omega$ and $L^\vee [a\sigma] (O_k\setminus Q_k)$ lies in the complement of $\Omega$, the above quantity is zero. This shows that $\Bo  u \in \mathring{V}_h^{k-1}$.

We next show \eqref{eq:Bog}. It suffices to show that
\begin{equation}
\int_{\tau} \Bo d u+  \int_{\tau} d \Bo u=\int_{\tau} u \quad \forall \tau \in X^{(k)}.
\end{equation} 

Hence, we compute for $1\leq k\leq n-1$,
\begin{alignat*}{2}
\int_{\tau} \Bo d u+  \int_{\tau} d \Bo u &= \int_{\Co^\Ss(\tau)} d u - \int_{\Co^\vee(\tau)} du +  \int_{\partial \tau} \Bo u \qquad && \text{by \eqref{def:Bo}, \eqref{eq:StokesSimp}}  \\
&= \int_{\partial^\Ss \Co^\Ss(\tau)} u - \int_{\partial^{\Ss^\vee} \Co^\vee(\tau)} u + \int_{\partial \tau} \Bo u \qquad && \text{by \eqref{eq:StokesSing}, \eqref{eq:StokesSingC}} \\
&= \int_{\partial^\Ss \Co^\Ss(\tau)} u - \int_{\partial^{\Ss^\vee} \Co^\vee(\tau)} u + \int_{\Co^\Ss(\partial \tau)} u - \int_{\Co^\vee(\partial \tau)} u \qquad && \text{by \eqref{def:Bo}} \\
&= \int_{L\tau} u = \int_{\tau} u. \qquad && \text{by \eqref{eq:hom1sing}, \eqref{hom2a}}.
\end{alignat*}
On the other hand, for $u\in V_h^0$ and $\tau\in X_0$, we have
\begin{equation*}
    \int_{\tau} \Bo d u = \int_{\Co^\Ss(\tau)} d u - \int_{\Co^\vee(\tau)} du = \int_{\partial^\Ss \Co^\Ss(\tau)} u - \int_{\partial^{\Ss^\vee} \Co^\vee(\tau)}= u(\tau),
\end{equation*}
where in the last step we used \eqref{eq:hom2sing} and \eqref{hom2b}.

\end{proof}

In the $k=n$ case, the operator in \eqref{def:Bo} no longer satisfies \eqref{eq:Bog} if $a$ lies in the boundary of $\tau$ for some $\tau\in X^{(n)}$. One can see this even when $\Omega$ is one-dimensional and $a$ is chosen to be a vertex in $X$. Hence, we need the following additional assumption.
\begin{assumption}\label{assump:a}
    We assume $a$ does not lie in any $\sigma\in X^{(n-1)}$.
\end{assumption}

\begin{remark}
    Assumption \ref{assump:a} means that $a$ does not lie on any facet of the simplicial complex. In practice, if one has $a$ lying on a facet, a simple local perturbation of the simplicial complex will again allow for Assumption \ref{assump:a} to hold.
\end{remark}

The key to the $k=n$ case are the following two lemmas below. The proof of Lemma \ref{lem:k=nkey} can be found in Appendix \ref{app:k=nlem}.

\begin{lemma}\label{lem:k=nkey}
    If $a$ is as in Assumption \ref{assump:a}, then for $\omega \in L^1\Lambda^n(\R^n)$, with compact support and $\tau\in X^{(n)}$,
    \begin{equation}\label{eq:intresult}
        \int_{\Co^\vee(\partial\tau)} \omega = \chi_\tau(a) \int_{\R^n} \omega.
    \end{equation}
    where $\chi_\tau$ is the indicator function on $\tau$. That is, $\chi_\tau(x) = 1$ for all $x\in\tau$ and $\chi_\tau(x) = 0$ otherwise.
\end{lemma}

\begin{lemma}
    For $\omega\in L^1\La^n(\Omega)$ and $\tau\in X^{(n)}$, it holds that
    \begin{equation}\label{eq:con}
        \int_{\Co^\Ss(\partial \tau)} \omega = \int_\tau \omega.
    \end{equation}
\end{lemma}
\begin{proof}
    Recall that by \eqref{eq:hom1sing}, $\Co^\Ss(\partial \tau) + \partial^\Ss \Co^\Ss(\tau) = L\tau$. Hence, for any $\tau\in X^{(n)}$, by \eqref{eq:StokesSing},
    \begin{equation*}
        \int_{\Co^\Ss(\partial \tau)} \omega = \int_\tau \omega - \int_{\partial^\Ss \Co^\Ss(\tau)} \omega = \int_\tau \omega - \int_{\Co^\Ss(\tau)}  d\omega = \int_\tau \omega
    \end{equation*}
    where we used that $d = 0$ on $n$-forms in the last equality.
\end{proof}

\begin{theorem}\label{thm:Bog2}
    Under Assumption \ref{assump:a}, Theorem \ref{thm:Bog} still holds when $k=n$.
\end{theorem}

\begin{proof}
    The fact that $\mathcal{B}u\in \mathring{V}_h^{n-1}$ follows from the same proof as that of Theorem \ref{thm:Bog}. It remains to show \eqref{eq:Bog}. To this end, we compute
\begin{align*}
    \int_\tau d\Bo u = \int_{\partial \tau} \Bo u = \int_{\Co^\Ss(\partial\tau)} u - \int_{\Co^\vee(\partial\tau)} u = \int_\tau u - \chi_\tau(a)\int_{\R^n} u 
\end{align*}
where in the last step we used \eqref{eq:con} and \eqref{eq:intresult} (and Assumption \ref{assump:a}). Then, since $u\in \mathring{V}_h^n$, it holds that $\int_{\R^n} u = \int_\Omega u = 0$,
and so we are done.
\end{proof}

\begin{remark}
    As in Definition \ref{Whitney_cochain_op}, one can also define the discrete Bogovski\u{\i} operator on cochains using the de Rham and Whitney maps. The proofs of the relevant properties follow the same structure as in Corollary \ref{cor:Whitney_cochain_op}.
\end{remark}

\section{Numerical Experiments}\label{sec:numerical_experiments}

We provide numerical experiments illustrating the application of the discrete Poincar\'e and Bogovski\u{\i} operators defined above. We first show examples of the Whitney form based and purely combinatorial discrete Poincar\'e operators on a simplicial complex triangulating the unit square using the cone operators of Section \ref{sec:cones_special_cases}. Following this, we show experiments of the two operators on a domain which is not star-shaped with respect to a point using the cone operators of Section \ref{sec:cones_general}. Finally, we show experiments for the discrete Bogovski\u{\i} operator on the simplicial complex triangulating the unit square. Illustrations of the two simplicial complexes we consider are in Figure \ref{fig:meshes} below.

\begin{figure}[htbp]
    \centering
    \begin{subfigure}[b]{0.45\textwidth}
        \centering
        \includegraphics[width=0.5\linewidth]{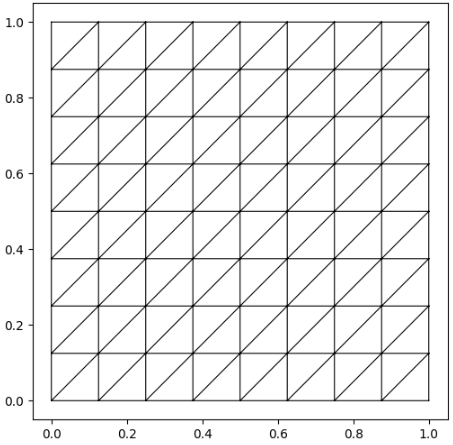}
        \caption{Simplicial complex triangulating the unit square}
        \label{fig:squaremesh}
    \end{subfigure}
    \hfill
    \begin{subfigure}[b]{0.45\textwidth}
        \centering
        \includegraphics[width=0.5\linewidth]{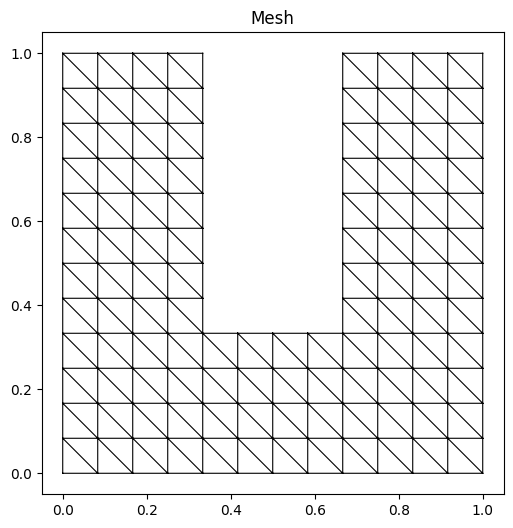}
        \caption{Simplicial complex triangulating a U-shaped domain}
        \label{fig:Umesh}
    \end{subfigure}
    \caption{Simplicial complexes triangulating a star-shaped domain and a non-star-shaped domain.}\label{fig:meshes}
\end{figure}

\subsection{Star-Shaped Domain}

We first consider $X$ to be a simplicial complex triangulating the unit square $[0,1]^2$ with a structured grid of $(N+1) \times (N+1)$ vertices. This simplicial complex is depicted in Figure \ref{fig:squaremesh} for $N=8$.

Since the domain is convex, we use the Poincar\'e operator obtained from the cone operator of Definition \ref{def:sing_cone_star}. For the operator acting on Whitney $1$-forms, the necessary line integrals are computed using simple quadrature formulas. For the operator acting on Whitney $2$-forms, the integrals are computed as a sum of the constant values on each $2$-simplex multiplied by the regions obtained from the intersections of $[a,\sigma]$ with triangles in the simplicial complex. These intersection areas are found using the Sutherland-Hodgman clipping algorithm \cite{SutherlandHodgman74}. Additionally, since the simplicial complex is simple, a collapse sequence can be easily found and so we also provide results for the cone operator built from a collapse sequence of Definition \ref{def:collapse_cone}.

To test the operators, we generated 100 random $k$-cochains for each $k\in\{0,1,2\}$. We then compute $P^1 d^0 u$ for $u\in \C^0$, $(d^0 P^1 + P^2 d^1) v$ for $v\in \C^1$, and $d^1 P^2 u$ for $w\in \C^2$. We then check the $\ell^\infty$ norm of the difference in cochain values between these quantities and $u$, $v$, and $w$ respectively (after subtracting the requisite global constant in the $k=0$ case). We see machine precision in all cases.

Now, we showcase the two operators' abilities to compute discrete 2d vector potentials. We consider $g(x,y) = x(1-x)y(1-y)$,
and let $g_h = \Pi g$ where $\Pi$ is the canonical projection onto the piecewise constant space. We find discrete vector potentials using the discrete Poincar\'e operators as ${\bf v}_h = \Po g_h$ and ${\bf w}_h = W P \dr g_h$. The output of this experiment is illustrated in Figure \ref{fig:potentials_1}.

\begin{figure}[htbp]
    \centering
    \includegraphics[width=0.8\linewidth]{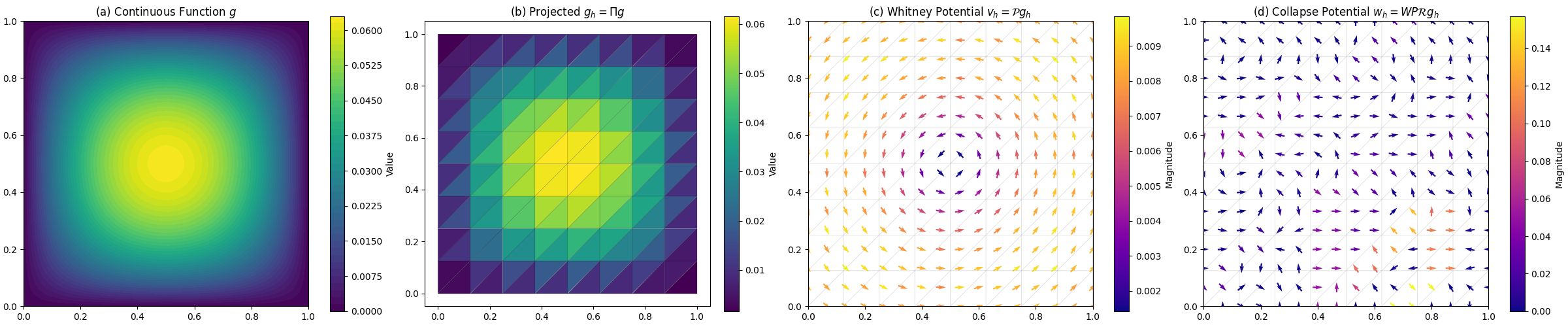}
    \caption{Discrete vector potentials on a star-shaped domain. We compute a potential using the Whitney form based discrete Poincar\'e operator using the generalized cone operator of Definition \ref{def:sing_cone_star} with point $a=(0.5,0.5)$ and the combinatorial discrete Poincar\'e operator using the collapse sequence based cone operator of Definition \ref{def:collapse_cone} where the complex is collapsing to the vertex at $(1.0,1.0)$. }
    \label{fig:potentials_1}
\end{figure}

\subsection{Non-Star-Shaped Domain}

We now consider $X$ to be a simplicial complex triangulating a U-shaped domain which is not star-shaped. This simplicial complex is depicted in Figure \ref{fig:Umesh} above.

To define the Whitney form based discrete Poincar\'e operator, we use the cone operator of Definition \ref{def:Pgen}. The requisite Lipschitz contraction is obtained by choosing a vertex $a = (a_0, a_1) \in [0.0, 1.0]\times[0.0, 0.3]$, and defining
\begin{equation}
    \Phi(x_1, x_2, t) = 
\begin{cases} 
\Big( (1 - 2t)a_1 + 2t x_1, \, a_2 \Big) & \text{for } t \in \left[0, \frac{1}{2}\right] \\
\Big( x_1, \, 2(1 - t)a_2 + (2t - 1)x_2 \Big) & \text{for } t \in \left(\frac{1}{2}, 1\right]
\end{cases}.
\end{equation}
This contraction essentially retracts any point in the U-shaped domain to the line $x_2 = a_2$ while $t \in (1/2, 1]$, and contracts this line to the point $(a_1, a_2)$ while $t \in [0, 1/2]$. Due to this specific contraction and mesh, the requisite integration for 1-forms follows mesh edges, while the requisite integration for 2-forms only involves the region corresponding to $t \in (1/2, 1]$.

A strong collapse sequence can also be easily found on this mesh and so the discrete Poincar\'e operator arising from the cone operator of Definition \ref{def:gen_simp_cone} with discrete contraction $\Psi$ built from the strong collapse sequence following Definition \ref{def:discrete_contraction}. We remark that a collapse sequence can also be easily derived on this simplicial complex, but our aim here is to showcase the discrete contraction based operator and so we use the strong collapse sequence.

We again test the discrete Poincar\'e operators by generating random $k$-cochains and checking the homotopy identities \eqref{eq:homotopy_formula} by computing the $\ell^\infty$ error as in the previous test. We see machine precision in all cases.

We also showcase the two operators' abilities to compute discrete scalar potentials in 2d. We consider the vector field $f(x,y)=(y-1/2, x-1/2)$ and let $f_h=\Pi f$ where $\Pi$ is the canonical projection onto the lowest order N\'ed\'elec space \cite{Nedelec1986}. We find discrete scalar potentials using the discrete Poincar\'e operators as $\phi_h^W = \Po f_h$ and $\phi_h^C = W P \dr f_h$. The output of this experiment is illustrated in Figure \ref{fig:potentials_2}.

\begin{figure}[htbp]
    \centering
    \includegraphics[width=0.8\linewidth]{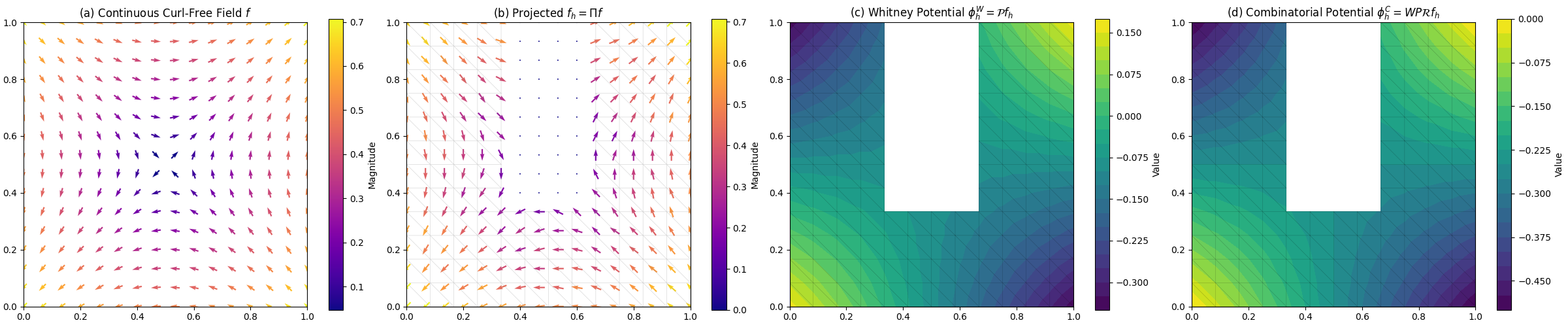}
    \caption{Discrete scalar potentials on a non-star-shaped domain. We compute a potential using the Whitney form based discrete Poincar\'e operator using the generalized cone operator of Definition \ref{def:Pgen} with point $a=(0.2,0.2)$ and the combinatorial discrete Poincar\'e operator using the discrete contraction based cone operator of Definition \ref{def:gen_simp_cone} built from a strong collapse sequence to the vertex at $(0.0,0.0)$.}
    \label{fig:potentials_2}
\end{figure}

\subsection{Discrete Bogovsk\u{\i} Operator}

We now again consider the simplicial complex triangulating the unit square depicted in Figure \ref{fig:squaremesh}. 

To compute the discrete Bogovski\u{\i}, we need to compute integrals over infinite cones. This is done as follows. For each vertex, $v$, in the simplicial complex, we associate a point $v^*$ in $\R^2$ which lies sufficiently far outside of the unit square on the ray starting at $a$ and pointing in the direction of $v$. Then, for the operator acting on Whitney $1$-forms, we use quadrature to integrate over the line segment between $a$ and $v^*$ as a proxy for the infinite cone. For the operator acting on Whitney $2$-forms, for each edge $e=[v_1, v_2]$ we integrate over the triangle $[a, v_1^*, v_2^*]$ as a proxy for the infinite cone. To perform this integration, we again use the clipping method noted for computing the integrals for the discrete Poincar\'e operator on star-shaped domains.

To test the discrete Bogovsk\u{\i} operators, we again generated 100 random $k$-cochains for each $k\in\{0,1,2\}$. We then set the cochain values associated to boundary simplices to zero in the $k=0$ and $k=1$ cases and in the $k=2$ case, we subtracted the global average of the corresponding Whitney $2$-form. We then verified the homotopy identities \eqref{eq:homotopy_formula} by checking the $\ell^\infty$ error as in previous tests. We see machine precision in all cases.

Now, we showcase the operator's abilities to compute discrete 2d vector potentials with boundary conditions. We consider $g(x,y) = x(1-x)y(1-y) - 1/36$ (which has average zero) and let $g_h = \Pi g$ where $\Pi$ is the canonical projection onto the piecewise constant space. We note that $g$ has average zero. We find discrete vector potential using the discrete Bogovski\u{\i} operator as ${\bf v}_h = \Bo g_h$. The output of this experiment is illustrated in Figure \ref{fig:potentials_3}.

\begin{figure}[htbp]
    \centering
    \includegraphics[width=0.7\linewidth]{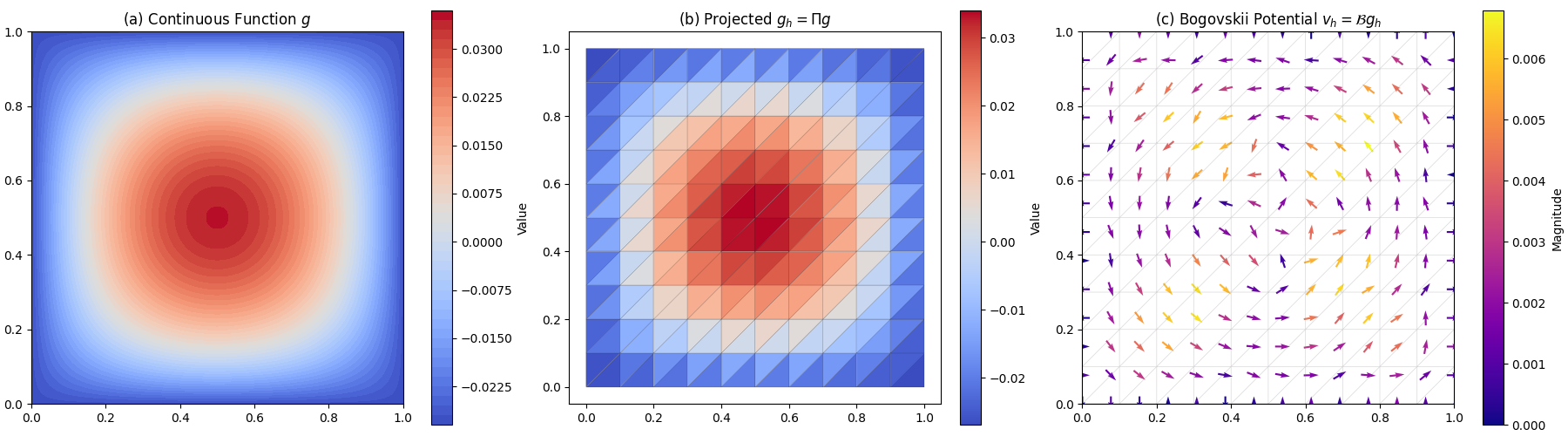}
    \caption{Discrete vector potentials on a star-shaped domain with boundary conditions. We compute a potential using the discrete Bogovski\u{\i} operator of Section \ref{sec:Bog} with point $a=(0.52,0.51)$ (this point does not lie on a mesh facet as required by Assumption \ref{assump:a}).}
    \label{fig:potentials_3}
\end{figure}

\bibliographystyle{abbrv}
\bibliography{references}

@book{AbMaRa1988,
	author = {Abraham, R. and Marsden, J. E. and Ratiu, T.},
	edition = {Second},
	publisher = {Springer-Verlag, New York},
	title = {Manifolds, tensor analysis, and applications},
	year = {1988}
}

@article{Dodziuk1976,
	author = {Dodziuk, Jozef},
	fjournal = {American Journal of Mathematics},
	journal = {Amer. J. Math.},
	number = {1},
	pages = {79--104},
	title = {Finite-difference approach to the {H}odge theory of harmonic forms},
	volume = {98},
	year = {1976}
}

@book {DeRaSa2010,
    AUTHOR = {De Loera, Jes\'us A. and Rambau, J\"org and Santos, Francisco},
     TITLE = {Triangulations: {S}tructures for algorithms and applications},
 PUBLISHER = {Springer-Verlag, Berlin},
      YEAR = {2010},
       DOI = {10.1007/978-3-642-12971-1}}

@article{HeHi2011,
	author = {Holger Heumann and Ralf Hiptmair},
	journal = {Discrete and Continuous Dynamical Systems},
	number = {4},
	pages = {1471--1495},
	title = {{E}ulerian and semi-{L}agrangian methods for convection-diffusion for differential forms},
	volume = {29},
	year = {2011},
	doi = {10.3934/dcds.2011.29.1471}
}

@article{ArnoldFalkWinther2006,
  title={Finite element exterior calculus, homological techniques, and applications},
  author={Arnold, Douglas N and Falk, Richard S and Winther, Ragnar},
  journal={Acta Numerica},
  volume={15},
  pages={1--155},
  year={2006},
  publisher={Cambridge University Press}
}

@book {Whitney1957,
    AUTHOR = {Whitney, Hassler},
     TITLE = {Geometric integration theory},
 PUBLISHER = {Princeton University Press, Princeton, NJ},
      YEAR = {1957},
     PAGES = {xv+387},
   MRCLASS = {53.0X},
  MRNUMBER = {87148},
MRREVIEWER = {H.\ Samelson},
}

@article{Nedelec1986,
	Author = {N{\'e}d{\'e}lec, J.-C.},
	Journal = {Numer. Math.},
	Number = {1},
	Pages = {57--81},
	Title = {A new family of mixed finite elements in {${\bf R}\sp 3$}},
	Volume = {50},
	Year = {1986}
}

@article{Desbrun2005,
title = {Discrete Poincaré lemma},
journal = {Applied Numerical Mathematics},
volume = {53},
number = {2},
pages = {231-248},
year = {2005},
note = {Tenth Seminar on Numerical Solution of Differential and Differntial-Algebraic Euqations (NUMDIFF-10)},
issn = {0168-9274},
doi = {https://doi.org/10.1016/j.apnum.2004.09.035},
url = {https://www.sciencedirect.com/science/article/pii/S0168927404002119},
author = {Mathieu Desbrun and Melvin Leok and Jerrold E. Marsden}
}

@book {Munkres1984,
    AUTHOR = {Munkres, James R.},
     TITLE = {Elements of algebraic topology},
 PUBLISHER = {Addison-Wesley Publishing Company, Menlo Park, CA},
      YEAR = {1984},
     PAGES = {ix+454},
      ISBN = {0-201-04586-9},
   MRCLASS = {55-01},
  MRNUMBER = {755006},
MRREVIEWER = {Christopher\ W.\ Stark},
}

@article {Costabel2010,
    AUTHOR = {Costabel, Martin and McIntosh, Alan},
     TITLE = {On {B}ogovski\u i\ and regularized {P}oincar\'e{} integral
              operators for de {R}ham complexes on {L}ipschitz domains},
   JOURNAL = {Math. Z.},
  FJOURNAL = {Mathematische Zeitschrift},
    VOLUME = {265},
      YEAR = {2010},
    NUMBER = {2},
     PAGES = {297--320},
      ISSN = {0025-5874,1432-1823},
   MRCLASS = {58J10 (35B65 35S05 47G30)},
  MRNUMBER = {2609313},
MRREVIEWER = {Horst\ Heck},
       DOI = {10.1007/s00209-009-0517-8},
       URL = {https://doi.org/10.1007/s00209-009-0517-8},
}

@book {Lee2013,
    AUTHOR = {Lee, John M.},
     TITLE = {Introduction to smooth manifolds},
    SERIES = {Graduate Texts in Mathematics},
    VOLUME = {218},
   EDITION = {Second},
 PUBLISHER = {Springer, New York},
      YEAR = {2013},
     PAGES = {xvi+708},
      ISBN = {978-1-4419-9981-8},
   MRCLASS = {58-01 (53-01 57-01)},
  MRNUMBER = {2954043},
}

@article {ChristiansenHu2018,
    AUTHOR = {Christiansen, Snorre H. and Hu, Kaibo},
     TITLE = {Generalized finite element systems for smooth differential
              forms and {S}tokes' problem},
   JOURNAL = {Numer. Math.},
  FJOURNAL = {Numerische Mathematik},
    VOLUME = {140},
      YEAR = {2018},
    NUMBER = {2},
     PAGES = {327--371},
      ISSN = {0029-599X,0945-3245},
   MRCLASS = {65N30 (58A12 76D07)},
  MRNUMBER = {3851060},
MRREVIEWER = {J\'anos\ Kar\'atson},
       DOI = {10.1007/s00211-018-0970-6},
       URL = {https://doi.org/10.1007/s00211-018-0970-6},
}

@article {Hiptmair1999,
    AUTHOR = {Hiptmair, R.},
     TITLE = {Canonical construction of finite elements},
   JOURNAL = {Math. Comp.},
  FJOURNAL = {Mathematics of Computation},
    VOLUME = {68},
      YEAR = {1999},
    NUMBER = {228},
     PAGES = {1325--1346},
      ISSN = {0025-5718,1088-6842},
   MRCLASS = {65N30 (41A30)},
  MRNUMBER = {1665954},
MRREVIEWER = {Stefan\ A.\ Funken},
       DOI = {10.1090/S0025-5718-99-01166-7},
       URL = {https://doi.org/10.1090/S0025-5718-99-01166-7},
}

@article{Bossavit2003,
	author = {Alain Bossavit},
	journal = {COMPEL : The International Journal for Computation and Mathematics in Electrical and Electronic Engineering},
	number = {3},
	pages = {470--480},
	title = {Extrusion, contraction : their discretization via Whitney forms},
	volume = {22},
	year = {2003}}

@phdthesis{Hirani2003,
	author = {Anil N. Hirani},
	month = {5},
	school = {California Institute of Technology},
	title = {Discrete Exterior Calculus},
	year = {2003}}

@article{DoMoSh2008,
  author    = {V. V. Dolotin and A. Yu. Morozov and Sh. R. Shakirov},
  title     = {An ${A}_\infty$ Structure on Simplicial Complexes},
  journal   = {Theoretical and Mathematical Physics},
  volume    = {156},
  number    = {1},
  pages     = {965--995},
  year      = {2008},
  doi       = {10.1007/s11232-008-0074-8},
  note      = {Originally published in \emph{Teoreticheskaya i Matematicheskaya Fizika}, 156(1):110–145 (2008)}
}

@phdthesis{Liu2026,
    author = {Bingyan Liu},
    title = {Discrete {Poincar\'e} Operators and {A}-Infinity Structures in Discrete Exterior Calculus},
    school = {University of Illinois at Urbana--Champaign},
    year = {2026}
}

@article {Demkowicz2005,
    AUTHOR = {Demkowicz, L. and Buffa, A.},
     TITLE = {{$H^1$}, {$H({\rm curl})$} and {$H({\rm div})$}-conforming
              projection-based interpolation in three dimensions.
              {Q}uasi-optimal {$p$}-interpolation estimates},
   JOURNAL = {Comput. Methods Appl. Mech. Engrg.},
  FJOURNAL = {Computer Methods in Applied Mechanics and Engineering},
    VOLUME = {194},
      YEAR = {2005},
    NUMBER = {2-5},
     PAGES = {267--296},
      ISSN = {0045-7825,1879-2138},
   MRCLASS = {65N30 (78M10)},
  MRNUMBER = {2105164},
MRREVIEWER = {JiChun\ Li},
       DOI = {10.1016/j.cma.2004.07.007},
       URL = {https://doi-org.revproxy.brown.edu/10.1016/j.cma.2004.07.007},
}

@article {Demkowicz2003,
    AUTHOR = {Demkowicz, L. and Babu{\v s}ka, I.},
     TITLE = {{$p$} interpolation error estimates for edge finite elements
              of variable order in two dimensions},
   JOURNAL = {SIAM J. Numer. Anal.},
  FJOURNAL = {SIAM Journal on Numerical Analysis},
    VOLUME = {41},
      YEAR = {2003},
    NUMBER = {4},
     PAGES = {1195--1208},
      ISSN = {0036-1429,1095-7170},
   MRCLASS = {65N30 (35Q60 65N15 78A25 78M10)},
  MRNUMBER = {2034876},
MRREVIEWER = {Daniele\ Boffi},
       DOI = {10.1137/S0036142901387932},
       URL = {https://doi-org.revproxy.brown.edu/10.1137/S0036142901387932},
}

@article {CapHu2023,
    AUTHOR = {\v{C}ap, Andreas and Hu, Kaibo},
     TITLE = {Bounded {P}oincar\'e{} operators for twisted and {BGG}
              complexes},
   JOURNAL = {J. Math. Pures Appl. (9)},
  FJOURNAL = {Journal de Math\'ematiques Pures et Appliqu\'ees. Neuvi\`eme
              S\'erie},
    VOLUME = {179},
      YEAR = {2023},
     PAGES = {253--276},
      ISSN = {0021-7824,1776-3371},
   MRCLASS = {58J10 (58A10 65N30)},
  MRNUMBER = {4659286},
MRREVIEWER = {Giovanni\ Moreno},
       DOI = {10.1016/j.matpur.2023.09.008},
       URL = {https://doi-org.revproxy.brown.edu/10.1016/j.matpur.2023.09.008},
}

@article {Bogovskii1979,
    AUTHOR = {Bogovski{\u i}, M. E.},
     TITLE = {Solution of the first boundary value problem for an equation
              of continuity of an incompressible medium},
   JOURNAL = {Dokl. Akad. Nauk SSSR},
  FJOURNAL = {Doklady Akademii Nauk SSSR},
    VOLUME = {248},
      YEAR = {1979},
    NUMBER = {5},
     PAGES = {1037--1040},
      ISSN = {0002-3264},
   MRCLASS = {35Q20 (76N99)},
  MRNUMBER = {553920},
MRREVIEWER = {N.\ D.\ Kopachevski\u i},
}

@book {Hatcher2002,
    AUTHOR = {Hatcher, Allen},
     TITLE = {Algebraic topology},
 PUBLISHER = {Cambridge University Press, Cambridge},
      YEAR = {2002},
     PAGES = {xii+544},
      ISBN = {0-521-79160-X; 0-521-79540-0},
   MRCLASS = {55-01 (55-00)},
  MRNUMBER = {1867354},
MRREVIEWER = {Donald\ W.\ Kahn},
}

@book {Leok2004,
    AUTHOR = {Leok, Melvin},
     TITLE = {Foundations of computational geometric mechanics},
      NOTE = {Thesis (Ph.D.)--California Institute of Technology},
 PUBLISHER = {ProQuest LLC, Ann Arbor, MI},
      YEAR = {2004},
     PAGES = {261},
      ISBN = {978-0496-07066-4},
   MRCLASS = {99-05},
  MRNUMBER = {2706498},
       URL =
              {http://gateway.proquest.com.revproxy.brown.edu/openurl?url_ver=Z39.88-2004&rft_val_fmt=info:ofi/fmt:kev:mtx:dissertation&res_dat=xri:pqdiss&rft_dat=xri:pqdiss:3147961},
}

@article {Matousek08,
    AUTHOR = {Matou{\v s}ek, Ji{\v r}{\'i}},
     TITLE = {L{C} reductions yield isomorphic simplicial complexes},
   JOURNAL = {Contrib. Discrete Math.},
  FJOURNAL = {Contributions to Discrete Mathematics},
    VOLUME = {3},
      YEAR = {2008},
    NUMBER = {2},
     PAGES = {37--39},
      ISSN = {1715-0868},
   MRCLASS = {55U10 (68R99)},
  MRNUMBER = {2455228},
MRREVIEWER = {Jakob\ Jonsson},
}

@article {Barmak2012,
    AUTHOR = {Barmak, Jonathan Ariel and Minian, Elias Gabriel},
     TITLE = {Strong homotopy types, nerves and collapses},
   JOURNAL = {Discrete Comput. Geom.},
  FJOURNAL = {Discrete \& Computational Geometry. An International Journal
              of Mathematics and Computer Science},
    VOLUME = {47},
      YEAR = {2012},
    NUMBER = {2},
     PAGES = {301--328},
      ISSN = {0179-5376,1432-0444},
   MRCLASS = {05E45 (57Q10)},
  MRNUMBER = {2872540},
MRREVIEWER = {John\ F.\ Oprea},
       DOI = {10.1007/s00454-011-9357-5},
       URL = {https://doi-org.revproxy.brown.edu/10.1007/s00454-011-9357-5},
}

@article {Stasheff1963,
    AUTHOR = {Stasheff, James Dillon},
     TITLE = {Homotopy associativity of {$H$}-spaces. {I}, {II}},
      NOTE = {{\bf 108} (1963), 275-292; ibid},
   JOURNAL = {Trans. Amer. Math. Soc.},
  FJOURNAL = {Transactions of the American Mathematical Society},
    VOLUME = {108},
      YEAR = {1963},
     PAGES = {293--312},
      ISSN = {0002-9947,1088-6850},
   MRCLASS = {55.40},
  MRNUMBER = {158400},
MRREVIEWER = {A.\ H.\ Clark},
       DOI = {10.1090/s0002-9947-1963-0158400-5},
       URL = {https://doi-org.revproxy.brown.edu/10.1090/s0002-9947-1963-0158400-5},
}

@inproceedings {Whitney1952,
    AUTHOR = {Whitney, Hassler},
     TITLE = {{$r$}-dimensional integration in {$n$}-space},
 BOOKTITLE = {Proceedings of the {I}nternational {C}ongress of
              {M}athematicians, {C}ambridge, {M}ass., 1950, vol. 1},
     PAGES = {245--256},
 PUBLISHER = {Amer. Math. Soc., Providence, RI},
      YEAR = {1952},
   MRCLASS = {27.2X},
  MRNUMBER = {43879},
MRREVIEWER = {T.\ Rad\'o},
}

@article{SutherlandHodgman74,
    author = {Sutherland, Ivan E. and Hodgman, Gary W.},
    title = {Reentrant Polygon Clipping},
    year = {1974},
    issue_date = {Jan. 1974},
    publisher = {Association for Computing Machinery},
    address = {New York, NY, USA},
    volume = {17},
    number = {1},
    issn = {0001-0782},
    journal = {Commun. ACM},
    pages = {32–42},
    numpages = {11},
}

@article {Farrell2011,
    AUTHOR = {Farrell, P. E. and Maddison, J. R.},
     TITLE = {Conservative interpolation between volume meshes by local
              {G}alerkin projection},
   JOURNAL = {Comput. Methods Appl. Mech. Engrg.},
  FJOURNAL = {Computer Methods in Applied Mechanics and Engineering},
    VOLUME = {200},
      YEAR = {2011},
    NUMBER = {1-4},
     PAGES = {89--100},
      ISSN = {0045-7825,1879-2138},
   MRCLASS = {65D05 (65D99 76M12)},
  MRNUMBER = {2740809},
       DOI = {10.1016/j.cma.2010.07.015},
       URL = {https://doi-org.revproxy.brown.edu/10.1016/j.cma.2010.07.015},
}

@article {Pitassi2022,
    AUTHOR = {Pitassi, Silvano and Ghiloni, Riccardo and Specogna, Ruben},
     TITLE = {Inverting the discrete curl operator: a novel graph algorithm
              to find a vector potential of a given vector field},
   JOURNAL = {J. Comput. Phys.},
  FJOURNAL = {Journal of Computational Physics},
    VOLUME = {466},
      YEAR = {2022},
     PAGES = {Paper No. 111404, 18},
      ISSN = {0021-9991,1090-2716},
   MRCLASS = {65D15 (65M06 65M60 65N06 65N30)},
  MRNUMBER = {4449581},
       DOI = {10.1016/j.jcp.2022.111404},
       URL = {https://doi-org.revproxy.brown.edu/10.1016/j.jcp.2022.111404},
}

@article {Whitehead1950,
    AUTHOR = {Whitehead, J. H. C.},
     TITLE = {Simple homotopy types},
   JOURNAL = {Amer. J. Math.},
  FJOURNAL = {American Journal of Mathematics},
    VOLUME = {72},
      YEAR = {1950},
     PAGES = {1--57},
      ISSN = {0002-9327,1080-6377},
   MRCLASS = {56.0X},
  MRNUMBER = {35437},
MRREVIEWER = {J.\ Dugundji},
       DOI = {10.2307/2372133},
       URL = {https://doi-org.revproxy.brown.edu/10.2307/2372133},
}

@article {LeMenachClenetPiriou1998,
  AUTHOR={Le Menach, Y. and Clenet, S. and Piriou, F.},
  JOURNAL={IEEE Transactions on Magnetics}, 
  TITLE={Determination and utilization of the source field in 3D magnetostatic problems}, 
  YEAR={1998},
  VOLUME={34},
  NUMBER={5},
  PAGES={2509-2512},
  DOI={10.1109/20.717578}
  }

@article {RodriguezValli2015,
    AUTHOR = {Alonso Rodr\'iguez, Ana and Valli, Alberto},
     TITLE = {Finite element potentials},
   JOURNAL = {Appl. Numer. Math.},
  FJOURNAL = {Applied Numerical Mathematics. An IMACS Journal},
    VOLUME = {95},
      YEAR = {2015},
     PAGES = {2--14},
      ISSN = {0168-9274,1873-5460},
   MRCLASS = {65N30},
  MRNUMBER = {3349682},
MRREVIEWER = {Annalisa\ Quaini},
       DOI = {10.1016/j.apnum.2014.05.014},
       URL = {https://doi-org.revproxy.brown.edu/10.1016/j.apnum.2014.05.014},
}

@unpublished{dipietro2025,
      title={Conforming lifting and adjoint consistency for the Discrete de Rham complex of differential forms}, 
      author={Daniele A. Di Pietro and Jérôme Droniou and Silvano Pitassi},
      year={2025},
      eprint={2509.21449},
      archivePrefix={arXiv},
      primaryClass={math.NA},
      url={https://arxiv.org/abs/2509.21449},
      note={arXiv preprint.}
}

@article {Rodriguez_magnetostatics_2013,
    AUTHOR = {Rodr\'iguez, Ana Alonso and Bertolazzi, Enrico and Ghiloni,
              Riccardo and Valli, Alberto},
     TITLE = {Construction of a finite element basis of the first de {R}ham
              cohomology group and numerical solution of 3{D} magnetostatic
              problems},
   JOURNAL = {SIAM J. Numer. Anal.},
  FJOURNAL = {SIAM Journal on Numerical Analysis},
    VOLUME = {51},
      YEAR = {2013},
    NUMBER = {4},
     PAGES = {2380--2402},
      ISSN = {0036-1429,1095-7170},
   MRCLASS = {65N30 (35J46 35Q60 65N12)},
  MRNUMBER = {3090156},
MRREVIEWER = {Marius\ Ghergu},
       DOI = {10.1137/120890648},
       URL = {https://doi-org.revproxy.brown.edu/10.1137/120890648},
}

@article {Rodriguez_eddycurrent_2015,
    AUTHOR = {Alonso Rodr\'iguez, Ana and Bertolazzi, Enrico and Ghiloni,
              Riccardo and Valli, Alberto},
     TITLE = {Finite element simulation of eddy current problems using
              magnetic scalar potentials},
   JOURNAL = {J. Comput. Phys.},
  FJOURNAL = {Journal of Computational Physics},
    VOLUME = {294},
      YEAR = {2015},
     PAGES = {503--523},
      ISSN = {0021-9991,1090-2716},
   MRCLASS = {65N30 (78A25)},
  MRNUMBER = {3343738},
       DOI = {10.1016/j.jcp.2015.03.060},
       URL = {https://doi-org.revproxy.brown.edu/10.1016/j.jcp.2015.03.060},
}

@article {ErnVohralik20,
    AUTHOR = {Ern, Alexandre and Vohral\'ik, Martin},
     TITLE = {Stable broken {$H^1$} and {$H({\rm div})$} polynomial
              extensions for polynomial-degree-robust potential and flux
              reconstruction in three space dimensions},
   JOURNAL = {Math. Comp.},
  FJOURNAL = {Mathematics of Computation},
    VOLUME = {89},
      YEAR = {2020},
    NUMBER = {322},
     PAGES = {551--594},
      ISSN = {0025-5718,1088-6842},
   MRCLASS = {65N15 (65N30 76M10)},
  MRNUMBER = {4044442},
MRREVIEWER = {Riccardo\ Sacco},
       DOI = {10.1090/mcom/3482},
       URL = {https://doi-org.revproxy.brown.edu/10.1090/mcom/3482},
}

\appendix

\section{On the derivation of the collapse-based simplicial cone operator}\label{app:Whitehead}
We provide some details on how the simplicial cone operator in Section \ref{sec:collapse_cone} follows naturally from \cite[Theorem 3]{Whitehead1950}. For brevity, we only consider the case of $k$-chains with $k>0$. We begin with a collapsible simplicial complex $X$ and collapse sequence $X= X_m \searrow X_{m-1} \searrow \cdots \searrow X_0 = a$ for $a\in X^0$. As in Section \ref{sec:collapse_cone}, for each $j\in \{1,\dots,m\}$, we let $\sigma_j$ and $\tau_j$ be such that $X_{j-1}:= X_j\setminus\{\sigma_j, \tau_j\}$ where $\tau_j$ is the free face. 

Now, we look at a single collapse $X_{j}\searrow X_{j-1}$ and let $p:=\dim(\sig_j)$ so $\dim(\tau_j)=p-1$. The set of two elements $\{\sigma_j,\tau_j\}$ do not form a simplicial complex, but we can still consider the set of $k$-chains over this set, denoted $\C_k(X_j\setminus X_{j-1})$. In particular, 
\begin{equation}\label{directsum}
    \C_k(X_j) = \C_k(X_{j-1})\oplus\C_k(X_j\setminus X_{j-1}).
\end{equation}
We now apply \cite[Theorem 3]{Whitehead1950} to say there exist chain homotopies $\eta_k^j : \C_k(X_j) \rightarrow \C_{k+1}(X_j)$ and ``deformation retractions'' $\kappa_k^j:\C_k(X_j)\to \C_{k}(X_{j-1})$ such that for all $k>0$,
\begin{subequations}
    \begin{alignat}{3}
        &\eta_{k-1}^j (\p_k u) +  \p_k \eta_k^j (u) = u - \kappa_k^j (u), \quad &&\forall u \in \C_k(X_j),\label{homtop}\\
        &\partial_k \kappa_k^j (u) = \kappa_{k-1}^j(\partial_k u), \quad &&\forall u \in \C_k(X_j),\\
        &\eta_k^j (u) = 0, \quad \kappa_k^j (u) = u, \quad &&\forall u\in \C_{k}(X_{j-1}), \label{homtop1}\\
        &\eta_k^j (u) \in \C_{k+1}(X_j\setminus X_{j-1}), \quad &&\forall u\in \C_k(X_j\setminus X_{j-1}). \label{homtop2}
    \end{alignat}
\end{subequations}
Actually \eqref{homtop2} does not follow directly from the statement of \cite[Theorem 3]{Whitehead1950}, but on inspection of the the proof, one sees that the property \eqref{homtop2} indeed holds.

Directly from the above properties of $\eta_k^j$ and $\kappa_k^j$, we in fact have that
\begin{alignat}{1}\label{aux110}
\kappa_{p-1}^j(\tau_j)=\tau_j-\p_{p} \sigma_j, \quad \kappa_{p}^j(\sigma_j)= 0 , \quad \eta_{p-1}^j(\tau_j)=\sigma_j, \quad \eta_{p}^j(\sigma_j)= 0. 
\end{alignat}

To show this, we start with $\eta_p^j(\sigma_j)$. From \eqref{homtop2}, we have have $\eta_p^j(\sigma_j) \in  \C_{p+1}(X_j\setminus X_{j-1})$ but $\C_{p+1}(X_j\setminus X_{j-1})= \{0\}$ so $\eta_p^j(\sigma_j)=0$. Next, we turn to $\kappa_p^j(\sigma_j)$. Using that $\eta_p^j(\sigma_j)=0$ along with \eqref{homtop}, we have
\begin{equation}\label{aux111}
\sigma_j- \kappa_{p}^j(\sigma_j) = \eta_{p-1}^j(\partial_p \sigma_j) =   \eta_{p-1}^j (\partial_p \sigma_j-\tau_j)+\eta_{p-1}^j(\tau_j)=\eta_{p-1}^j (\tau_j),
\end{equation}
where we used that $\partial_p \sigma_j-\tau_j \in \C_{p-1}(X_{j-1})$ and hence $\eta_{p-1}^j (\partial_p \sigma_j-\tau_j)=0$ from \eqref{homtop1}. Then, using \eqref{aux111}, we have that $\kappa_p^j(\sigma_j) = \sigma_j-\eta_{p-1}^j (\tau_j)$ which is in $\C_p(X_j\setminus X_{j-1})$ due to \eqref{homtop2}. However, we know that $\kappa_p^j(\sigma_j) \in \C_p(X_{j-1})$ so by \eqref{directsum}, it must be that $\kappa_p^j(\sigma_j)=0$. Moreover, from \eqref{aux111}, this means that $\eta_{p-1}^j (\tau_j)=\sigma_j$. 
Lastly, we consider $\kappa_{p-1}^j(\tau_j)$. Since $\partial_{p-1}\tau_j \in \C_{p-2}(X_{j-1})$, we have $\eta_{p-2}^j(\partial_{p-1}\tau_j)=0$. Therefore, from \eqref{homtop}, we see that $\tau_j-\kappa_{p-1}^j(\tau_j)= \partial_{p} \eta_{p-1}^j(\tau_j)=  \partial_{p} \sigma_j$ so $\kappa_{p-1}^j(\tau_j)= \tau_j-\p_{p} \sigma_j$. 

Finally for all $k>0$, we set $H_k^0:=0$ and define the operators $H_k^j:\C_k(X_j)\to \C_{k+1}(X_j)$ for $j>0$ recursively by 
\begin{equation}
    H_k^j(u) := \eta_k^j(u) + H_{k}^{j-1}(\kappa_k^j(u)).
\end{equation}
Using \eqref{homtop1}, we immediately see that $H_k^j(u) = H_{k}^{j-1}(u)$ for all $u\in \C_k(X_{j-1})$ for all $k>0$, and using \eqref{aux110}, we immediately see that $H_p^j(\sigma_j)= 0$ and $H_{p-1}^j(\tau_j)= \sigma_j - H_{p-1}^{j-1}(\partial_p \sigma_j-\tau_j)$. Comparing this to \eqref{eq:collapse_cone_def}, we see that the operator $H_k^m$ is exactly the simplicial cone operator of Definition \ref{def:collapse_cone}.

\section{Proof of Lemma \ref{prop:ChainHomotopy}}\label{app:ChainHomotopy}

\begin{proof}[Proof of Lemma \ref{prop:ChainHomotopy}]
As $(\sigma \times I_n)_{k+1} = \sum_{r=0}^{n-1} (\sigma \times [\frac{r}{n},\frac{r+1}{n}])_{k+1}$, it suffices to show that the proposition for $n=1$.
By linearity of $E$, it suffices to show the chain homotopy identity hold for an ordered simplex $\sigma= [v_0,\dots,v_k]$. Let $\sigma^j = [v_0,\dots,\hat{v_j},\dots, v_k]$, $\tau_i = [(v_0,0),\dots,(v_i,0)(v_i,1),\dots,(v_k,1)]$, and $\tau_i^j$ be the simplex obtained from $\tau_i$ by removing the $j$-th vertex, i.e.
\[
\tau_{i}^j =
\begin{cases}
[(v_0,0),\dots,(v_{j-1},0)(v_{j+1},0)\dots(v_i,0)(v_i,1),\dots,(v_k,1)], & j < i, \\
[(v_0,0),\dots,(v_{i-1},0)(v_i,1),\dots,(v_k,1)], & j = i, \\
[(v_0,0),\dots,(v_{i},0)(v_{i+1},1),\dots,(v_k,1)], & j = i+1, \\
[(v_0,0),\dots,(v_{i},0)(v_{i},1)\dots(v_{j-1},1)(v_{j+1},1),\dots,(v_k,1)],  & j > i+1.
\end{cases}
\]
Then we have
\begin{align*}
\partial E(\sigma) 
& = \partial \left(\sum_{i=0}^k (-1)^i \tau_i\right) 
= \sum_{i=0}^k (-1)^i \partial \tau_i  
= \sum_{i=0}^k (-1)^i \sum_{j=0}^{k+1}(-1)^j\tau_{i}^j
= \sum_{i=0}^k\sum_{j=0}^{k+1} (-1)^{i+j}\tau_i^j.
\end{align*}
A direct computation shows that 
\begin{align*}
E(\sigma^j) = \sum_{i=0}^{j-1}(-1)^i \tau_i^{j+1} + \sum_{i=j+1}^k (-1)^{i-1}\tau_i^j,
\end{align*}
thus we have  
\begin{align*}
E(\partial \sigma) 
& = E \left(\sum_{j=0}^k (-1)^j \sigma^j\right) 
= \sum_{j=0}^k (-1)^j E(\partial \sigma^j) 
= \sum_{j=0}^k (-1)^j \left(\sum_{i<j}(-1)^i \tau_i^{j+1} + \sum_{i>j}(-1)^{i-1} \tau_i^j\right) \\
& =\sum_{j=1}^{k+1}\sum_{i<j-1} (-1)^{i+j-1} \tau_i^j + \sum_{j=0}^k\sum_{i>j}(-1)^{i+j-1} \tau_i^j = \sum_{j=0}^{k+1} \sum_{i \neq j-1, i \neq j}(-1)^{i+j-1} \tau_i^j.
\end{align*}

Hence,
\begin{align*}
(\partial E + E \partial)(\sigma) 
& = \sum_{j=0}^{k+1}\sum_{i=\max(j-1,0)}^{\min(j,k)}  (-1)^{i+j}\tau_i^j 
= \tau_0^0 - \tau_{k}^{k+1} + \sum_{j=1}^k \left(\tau_{j}^{j} - \tau_{j-1}^j\right).
\end{align*}
Notice that $\tau_{j}^{j} = \tau_{j-1}^j$, $(j_1)_\#(\sigma) = \tau_0^0$ and $(j_0)_\#(\sigma) =\tau_{k}^{k+1}.$ This concludes the proof. 
\end{proof}

\section{Proof of Lemma \ref{lem:k=nkey}}\label{app:k=nlem}
First, we remark about the orientation of cones. For the linear singular $n$-cone $L^\vee$ corresponding to points $y_0, \dots, y_n$ in $\R^n$, the orientation of $L^\vee$ is $\mathcal{O}(L^\vee) := \Sign(\det([y_1-y_0, \dots, y_n-y_0]))$.

We first provide a general lemma related to the orientation of $\Co^\vee(f_i)$ for any face $f_i$ of $\tau\in X^n$. Then, we give two key lemmas which will give us everything we need to prove Lemma \ref{lem:k=nkey}.
\begin{lemma}\label{lem:orientationrelation}
    Let $\tau=[x_0, \dots, x_n]\in X^n$ and $f_i = [x_0, \dots, \widehat{x_i}, \dots, x_n]$. Let $a\in\R^n$ be arbitrary, $v_j := x_j - a$ for all $j\in \{0, \dots, n\}$, and let $B_i$ be the matrix $[v_0, \dots, \widehat{v_i} ,\dots, v_n]$. 
    
    Recall that we can express $a$ in terms of its barycentric coordinates as $a = \sum_{k=0}^n \lambda_k x_k$ where $\sum_{k=0}^n \lambda_k = 1$. For any $i,j\in\{0, \dots, n\}$ such that $\lambda_i, \lambda_j \neq 0$,
    \begin{equation}\label{eq:orientationrelation}
        \det(B_j) = \frac{\lambda_j}{\lambda_i} (-1)^{j-i} \det(B_i).
    \end{equation}
\end{lemma}

\begin{proof}
    First, notice that $a = \sum_{k=0}^n \lambda_k x_k = \sum_{k=0}^n \lambda_k (v_k + a) = \sum_{k=0}^n \lambda_k v_k + a$ so $\sum_{k=0}^n \lambda_k x_k = 0$ meaning that $\lambda_i v_i = - \sum_{k\neq i} \lambda_k v_k$.
    Therefore, supposing without loss of generality that $i < j$,
    \begin{align*}
        \det(B_j) & = -\frac{1}{\lambda_i} \sum_{k\neq i} \lambda_k \det \left[v_0, \dots, v_k, \dots, \widehat{v_{j}}, \dots, v_n\right]\\
        &= -\frac{\lambda_j}{\lambda_i} \det \left[v_0, \dots, v_j, \dots, \widehat{v_{j}}, \dots, v_n\right] = \frac{\lambda_j}{\lambda_i} (-1)^{j-i} \det(B_i).
    \end{align*}
\end{proof}

\begin{lemma}\label{lem:ain}
    Let $\tau=[x_0, \dots, x_n]\in X^n$ and $f_i = [x_0, \dots, \widehat{x_i}, \dots, x_n]$. Suppose $a\in \Int(\tau)$ where $\Int(\cdot)$ denotes the interior of $\cdot$. Then, the following statements hold.
    \begin{enumerate}[label=(\alph*)]
        \item $\bigcup_{i=0}^n\image(\Co^\vee(f_i)) = \R^n.$
        
        \item  $\image(\Co^\vee(f_i))\cap \image(\Co^\vee(f_j))$ has zero $n$-dimensional measure for all $i\neq j$. Moreover, if a lies in the same plane as $f_i$, $\image(\Co^\vee(f_i))$ also has zero $n$-dimensional measure.

    \item $(-1)^i\mathcal{O}(\Co^\vee( f_i)) = (-1)^j\mathcal{O}(\Co^\vee(f_j))$ for all $i,j\in\{0, \dots, n\}$ such that $a$ does not lie in the plane of $f_i$ or $f_j$.
    \end{enumerate}
\end{lemma}

\begin{proof}
    \emph{Proof of (a): }For any $p\in\R^n$, we denote the ray with initial point $a$ which passes through $p$ as $r(p)$. Explicitly, $r(p)= \{a+t(p-a):t\geq0\}$. Since $\tau$ is closed and bounded $r(p)$ intersects at least one face $f_i$ of $\tau$. Hence, every point $p\in\R^n$ must lie in at least one of the $\image(\Co^\vee(f_i))$.
    
    \emph{Proof of (b): }
    Since $\tau$ is convex, for any $p\in\R^n$, $r(p)$ has exactly one intersection with the boundary of $\tau$. This intersection will either lie in the interior of a single face $f_i$ or it will lie in $f_i \cap f_j = g_{ij}:=[x_0, \dots, \widehat{x_i}, \dots, \widehat{x_j}, \dots, x_n]$. Hence, if $p\in \image(\Co^\vee(f_i))\cap \image(\Co^\vee(f_j))$, then $p\in \image(\Co^\vee (g_{ij}))$ which is either an infinite $(n-1)$-cone or $(n-2)$-cone, and hence has $n$-measure zero. 

    On the other hand, if a lies in the same plane as $f_i$, $\image(\Co^\vee(f_i))$ is a degenerate $n$-simplex and so has zero $n$-dimensional measure.

    \emph{Proof of (c): }

    Now, since $a$ is in the interior of the $n$-simplex $\tau$, its barycentric coordinates $\lambda_0, \dots, \lambda_n$ are all positive. That is, $\sum_{j=0}^n \lambda_j = 1$ and $a = \sum_{\ell=0}^n \lambda_\ell x_\ell$ with $\lambda_i > 0$ for all $i\in\{0, \dots, n\}$. Therefore, from \eqref{eq:orientationrelation}, we have that for any $i, j\in\{0, \dots, n\}$ such that $\lambda_i, \lambda_j \neq 0$,
    \begin{equation*}
        \mathcal{O}(\Co^\vee(f_j)) = (-1)^{j-i} \mathcal{O}(\Co^\vee(f_j)).
    \end{equation*}
    The condition $\lambda_i, \lambda_j\neq 0$ is exactly when a does not lie in the same plane as $f_i$ and $f_j$ so we indeed have $(-1)^j \mathcal{O}(\Co^\vee(f_j)) = (-1)^{i} \mathcal{O}(\Co^\vee(f_i))$ in this case.
\end{proof}

\begin{lemma}\label{lem:aout}
    Let $\tau=[x_0, \dots, x_n]\in X^n$ and $f_i = [x_0, \dots, \widehat{x_i}, \dots, x_n]$. Suppose $a\in \R^n \setminus \tau$ and let $\lambda_0, \dots, \lambda_n$ be the barycentric coordinates of $a$. Then, the following statements hold.
    \begin{enumerate}[label=(\alph*)]
        \item Let $I = \{i: \lambda_i >0 \}$ and $J = \{j: \lambda_j < 0 \}$. It holds that $\bigcup_{i\in I}\image(\Co^\vee(f_i)) = \bigcup_{j\in J} \image(\Co^\vee(f_j))$, and $\image(\Co^\vee(f_k))$ has zero $n$-dimensional measure for all $k\notin I\cup J$.
        \item  $\image(\Co^\vee(f_i))\cap \image(\Co^\vee(f_j))$ has zero $n$-dimensional measure for all $i,j\in\{0, \dots, n\}$ such that $i\neq j$ and both $i,j \in I$ or both $i,j\notin I$. 
        \item Let $i, j \in \{0, \dots, n\}$ If $i,j \in I$ or $i,j\in J$, then $(-1)^i \mathcal{O}(\Co^\vee( f_i)) = (-1)^j\mathcal{O}(\Co^\vee(f_j))$. On the other hand, if $i \in I$ and $j\in J$, then $(-1)^i \mathcal{O}(\Co^\vee( f_i)) = -(-1)^j\mathcal{O}(\Co^\vee(f_j))$.
    \end{enumerate}
\end{lemma}
\begin{proof}

\emph{Proof of (a): } Note that since $a\in \R^n\setminus \tau$, both $I$ and $J$ must be nonempty. Now, for brevity, let $K_i = \image(\Co^\vee(f_i))$ and $v_i=x_i-a$ for any $i\in\{0, \dots, n\}$. $K_i$ has the representation $K_i = \left\{a + \sum_{k\neq i} s_k v_k: s_k\geq 0\right\}$. Hence,
\begin{equation}\label{eq:conichull}
    p\in K_i \iff p-a \in \left\{\sum_{k\neq i} s_k v_k : s_k\geq 0\right\}.
\end{equation}

Now, we will show that for any $q \in K_i$ for $i\in I$, there exists $j\in J$ such that $q \in K_j$. In fact, we will find such $j$. To this end, let $q\in K_i$ and define $w = q-a$. By \eqref{eq:conichull}, $w = \sum_{k\neq i} s_k v_k$ for some $\{s_k\}_{k\neq i}$ such that $s_k\geq 0$.  Moreover, we define $w(t) = w - tv_i$ for $t\geq 0$.

Recall that $\sum_{k=0}^n \lambda_k v_k = 0$ as noted in the proof of Lemma \ref{lem:orientationrelation}. Hence, $v_i = -\sum_{k\neq i}\frac{\lambda_k}{\lambda_i} v_k$ and
\begin{align*}
    w(t) = \left( \sum_{k \neq i} s_k v_k \right) - t \left( -\sum_{k \neq i} \frac{\lambda_k}{\lambda_i} v_k \right) = \sum_{k \neq i} \left( s_k + t \frac{\lambda_k}{\lambda_i} \right) v_k.
\end{align*}

Consider the coefficient $c_k(t) = s_k + t \frac{\lambda_k}{\lambda_i}$. If $k\in I\setminus\{i\}$, then $\frac{\lambda_k}{\lambda_i} > 0$ so $c_k(t) \geq s_k \geq 0$ for all $t\geq 0$. On the other hand, if $k\in J$, then $\lambda_k<0$ so $c_k(t) \leq s_k$ and $c_k(t)$ is decreasing with $t$. Moreover, for any $j\in J$, $c_j(t) = 0$ at $t_j = -s_j \frac{\lambda_i}{\lambda_j}$. Let $t^* = \min\{t_j:j\in J\}$ and let $j^*$ be the index which achieves this minimum.

Thus, at $t=t^*$, we have that $c_{j^*}(t) = 0$ and $c_k(t^*) \geq 0$ for all $k\in \{0, \dots, n\} \setminus{i}$. Therefore,
\begin{equation*}
    w = w(t^*) + t^*v_i = \sum_{k\neq i, j^*} c_k(t^*)v_k + t^*v_i.
\end{equation*}
We have just expressed $w=q-a$ as a nonnegative combination of the vectors $\{v_k\}_{k\neq j^*}$ so by \eqref{eq:conichull}, $q\in K_{j^*}$. This shows that $\bigcup_{i\in I}\image(\Co^\vee(f_i)) \subset \bigcup_{j\in J} \image(\Co^\vee(f_j))$. The reverse inclusion can be proven in a symmetric way.

Finally, if $k\notin I\cup J$, then $\lambda_k = 0$ meaning that $a$ lies on the same plane as $f_k$ and so $\image(\Co^\vee(f_k))$ is a degenerate $n$-simplex and so has zero $n$-dimensional measure.

\emph{Proof of (b): }This follows for a similar reason as in the proof of Lemma \ref{lem:ain}(b). First, let $i,j \in I$. Then, for any $p\in K_i \cap K_j$, we know that by (a), $p$ must also lie in $K_\ell$ where $\ell\notin I$. Therefore, $r(p)$ intersects both $f_\ell$ and $f_i \cap f_j = g_{ij}$ Thus, if $p\in \image(\Co^\vee(f_i))\cap \image(\Co^\vee(f_j))$, then $ p\in \image(\Co^\vee (g_{ij}))$ which is either an infinite $(n-1)$-cone or $(n-2)$-cone, and hence has $n$-measure zero. The case where $i,j \notin I$ is similar.

\emph{Proof of (c): }
First, suppose either $i, j \in I$ or $i,j \in J$. Then, $\frac{\lambda_j}{\lambda_i} > 0$ so the result follows from \eqref{eq:orientationrelation}. Similarly, if $i\in I$ and $j\in J$, then $\frac{\lambda_j}{\lambda_i} < 0$ so the result again follows from \eqref{eq:orientationrelation}.
\end{proof}
\begin{proof}[Proof of Lemma \ref{lem:k=nkey}]
    Recall that from our definition in \eqref{def:intSingC}, it holds that 
\begin{equation*}
    \int_{L^\vee} \omega = \int_{\image(L^\vee)} f dx_1\dots dx_n = \int_{\image(L^\vee)} f dV.
\end{equation*}
for an $n$-form $\omega = f \vol$ with compact support, and any linear singular cone $L^\vee$ (cf. \cite[Chapter 16]{Lee2013}). From this fact, one can see \eqref{eq:intresult} as follows. First, from Lemma \ref{lem:ain}, $a\in \Int(\tau)$,
\begin{equation*}
    \int_{\image(\Co^\vee(\partial\tau))} f dV = \int_{\R^n} f dV = \int_{\R^n} \omega.
\end{equation*}
Next, when $a\in \R^n \setminus \tau$, by Lemma \ref{lem:aout}, denoting $K_I = \bigcup_{i\in I} \image(\Co^\vee(f_i))$ and $K_J = \bigcup_{j\notin I} \image(\Co^\vee(f_j))$, we have
\begin{equation*}
    \int_{\image(\Co^\vee(\partial\tau))} f dV = \int_{K_I} f dV - \int_{K_J} f dV = 0.
\end{equation*}

\end{proof}

\end{document}